\documentclass[11pt,letterpaper]{amsart}

\usepackage{tikz}
\usepackage{comment}
\tikzstyle{every node}=[circle, draw, fill=black!50,
                        inner sep=0pt, minimum width=3pt]

\usepackage{amsfonts,amsmath,latexsym,color,epsfig,hyperref,enumitem, amssymb,bbm}

\newcommand{\conv}[1]{\operatorname{conv}#1}

\newtheorem{theorem}{Theorem}%
\newtheorem*{theorem*}{Theorem}
\newtheorem{lemma}{Lemma}%
\newtheorem{claim}[lemma]{Claim}
\newtheorem{definition}[lemma]{Definition}
\newtheorem{cor}[lemma]{Corollary}

\newtheorem{obs}[lemma]{Observation}

\renewcommand{\le}{\leqslant}
\renewcommand{\ge}{\geqslant}

\renewcommand{\geq}{\geqslant}

\def\qed{\ifvmode\mbox{ }\else\unskip\fi\hskip 1em plus 10fill$\Box$}

\input{epsf}

\makeatletter
\def\Ddots{\mathinner{\mkern1mu\raise\p@
\vbox{\kern7\p@\hbox{.}}\mkern2mu
\raise4\p@\hbox{.}\mkern2mu\raise7\p@\hbox{.}\mkern1mu}}
\makeatother

\def\R{\mathbb R}
\def\Z{\mathbb Z}

\def\N{\mathbb N}

\title{Sharp bound for the Erd\H{o}s--Straus non-averaging set problem}

\author{Huy Tuan Pham, Dmitrii Zakharov}
\thanks{Pham's research is supported by a Clay Research Fellowship.}
\address{School of Mathematics, Institute for Advanced Study, Princeton, NJ 08540, USA.}
\email{htpham@caltech.edu}
\thanks{Zakharov's research was supported by the Jane Street Graduate Fellowship.}
\address{Department of Mathematics, Massachusetts Institute of Technology, Cambridge, MA 02139, USA}
\email{zakhdm@mit.edu}

\date{}

\begin{document}

\begin{abstract}
A set of integers $A$ is {\em non-averaging} if there is no element $a$ in $A$ which can be written as an average of a subset of $A$ not containing $a$. We show that the largest non-averaging subset of $\{1, \ldots, n\}$ has size $n^{1/4+o(1)}$, thus solving the Erd{\H o}s--Straus problem. We also determine the largest size of a non-averaging set in a $d$-dimensional box for any fixed $d$. Our main tool includes the structure theorem for the set of subset sums due to Conlon, Fox and the first author, together with a result about the structure of a point set in nearly convex position.
\end{abstract}

\maketitle

\section{Introduction}

A set of integers $A$ is {\em non-averaging} if there is no element $a$ in $A$ which can be written as an average of a nonempty subset of $A$ not containing $a$. Alternatively, a non-averaging set $A$ is a set which avoids solutions in distinct variables to all equations of the form $kx_0=x_1+\dots+x_k$ for $k\ge 2$. The study of non-averaging sets was first initiated by Erd\H{o}s and Straus~\cite{Str,EStr} in the late 1960s, where they asked to determine the size of the largest non-averaging subset of $[n] := \{1, \ldots, n\}$. 

Let $h(n)$ denote the largest size of a non-averaging subset in $[n]$. In his original paper, Straus~\cite{Str} showed that $h(n) \geq e^{c \sqrt{\log n}}$ for some positive constant $c$, while, in a follow-up paper~\cite{EStr}, he and Erd\H{o}s showed that $h(n) = O(n^{2/3})$. The lower bound was improved to a polynomial by Abbott, who first showed~\cite{Abb1} that $h(n) = \Omega(n^{1/10})$ and then improved~\cite{Abb2} this bound to $h(n) = \Omega(n^{1/5})$. The current best lower bound, $h(n) = \Omega(n^{1/4})$, follows from a surprisingly simple construction due to Bosznay~\cite{Bosz}. Indeed, if we fix an integer $q$, then the set of integers consisting of $n_i = i q^3 + i(i+1)/2$ for $i = 1, 2, \dots, q-1$ is a non-averaging subset of $[n]$, where $n = q^4$. One can think of $A$ as a linear projection of the set of integer points on the parabola $S = \{(i, i^2), ~i=1,\ldots,q-1\} \subset [q] \times [q^2]$. The convexity of the parabola implies that $S$ is a non-averaging subset in $\Z^2$ and this property is preserved by an appropriate linear map. Another similar construction achieving the bound $h(n)=\Omega(n^{1/4})$ can be obtained by appropriately defining a set of points in convex position in $\mathbb{Z}^3$ together with a suitable linear map. Viewing a non-averaging set as an integer set avoiding certain linear configurations, this construction is along the line of a number of other important constructions in additive combinatorics based on convexity, such as Behrend's construction.

The study of non-averaging sets, since the work of Erd\H{o}s and Straus \cite{EStr}, has been closely tied to the study of additive structures in the set of subset sums. 
For a set $A$ let $\Sigma (A)$ denote the set of all subset sums of $A$, i.e. 
\[
\Sigma (A) = \left\{ \sum_{a \in B} a, ~ B \subseteq A \right\},
\]
where we include 0 in $\Sigma (A)$ as the sum of the empty subset sum. 
Erd\H os and Straus \cite{EStr} observed that if $A$ is a non-averaging set then for any $a \in A$ and any disjoint $A_1, A_2\subset A\setminus \{a\}$ we have
\begin{equation}\label{eq:sumsets_disjoint}
    \Sigma(A_1 - a) \cap \Sigma(a- A_2) = \{0\}.
\end{equation}
Indeed, suppose that for $B_i\subset A_i$ we have 
\[
\sum_{a_1 \in B_1} (a_1-a) = \sum_{a_2 \in B_2} (a-a_2)\neq 0
\]
then $B_1, B_2$ are non-empty and 
\[
a = \frac{\sum_{a_1\in B_1} a_1 + \sum_{a_2\in B_2} a_2}{|B_1| + |B_2|},
\]
so $A$ is not non-averaging. Note that the converse also holds: if $A$ is not non-averaging then (\ref{eq:sumsets_disjoint}) fails for some $a \in A$ and some disjoint $A_1 ,A_2 \subset A\setminus \{a\}$. 
Define $H(n)$ as the maximum integer for which there are two subsets of $[n]$ of size $H(n)$ whose sets of subset sums have no non-zero common element. The above argument implies that $h(n)\le 2H(n)+2$. Indeed, if $A = \{a_1, \ldots, a_m\}$ is a non-averaging set in $[n]$ then subsets $B_1 = \{a_{[m/2]}- a_1, \ldots, a_{[m/2]}- a_{[m/2]-1}\}$ and $B_2 = \{a_{[m/2]+1} - a_{[m/2]}, \ldots, a_m -a_{[m/2]} \}$ of $[n]$ have sets of subset sums disjoint apart from 0. 
Erd\H os and Straus \cite{EStr} proved that $H(n) = O(n^{2/3})$, which implies the same bound for $h(n)$. In 1990, Erd\H os and S\'ark\"ozy \cite{ES} improved the upper bound to $h(n) \le H(n) \ll (n \log n)^{1/2}$ based on a theorem of Freiman and S\'ark\"ozy on long homogeneous arithmetic progressions in the set of subset sums of subsets of $[n]$ of size at least $C\sqrt{n\log n}$. The key property to note here is that any two sufficiently long homogeneous arithmetic progressions in a bounded interval must intersect. Szemer\'edi and Vu \cite{SV} showed that the set of subset sums of subsets of $[n]$ of size at least $C\sqrt{n\log n}$ contains a long arithmetic progression, albeit without the homogeneity property that is important for the application to non-averaging sets. In \cite{CFP}, Conlon, Fox and the first author obtained a common strengthening of the result of Freiman and S\'ark\"ozy and the result of Szemer\'edi and Vu, and as a corollary obtained an improved upper bound $h(n) \le H(n) = O(\sqrt{n})$. This is a natural barrier for the problem of upper bounding $h(n)$, as we can no longer expect to find long arithmetic progressions among the subset sums for sets of size significantly less than $\sqrt{n}$. In fact, we have $H(n)=\Omega(\sqrt{n})$, since for $c$ sufficiently small, the sets $[c\sqrt{n}]$ and $[n-c\sqrt{n},n]$ have no non-zero common element among their subset sums. Very recently, Conlon, Fox and the first author \cite{conlon2023homogeneous} surpassed this barrier and showed $h(n)\le n^{\sqrt{2}-1+o(1)}$. 

The main result of this paper is the determination, up to lower order factors, of the maximum size of a non-averaging subset $A$ of $[n]$. 

\begin{theorem}\label{thm:main-0}
    Let $A \subseteq [n]$ be a non-averaging set. Then $|A| \le n^{1/4+o(1)}$. In particular, $h(n) = n^{1/4+o(1)}$. 
\end{theorem}

The sharp bound obtained in Theorem \ref{thm:main-0} is fairly surprising in that the constructions yielding the lower bound arise from point sets in convex positions, which are unstructured. In fact, the good competing lower bound constructions can take on significantly different shapes, being either projections from suitable $2$-dimensional or $3$-dimensional point sets.

The key result behind the improved bound on $h(n)$ in \cite{conlon2023homogeneous}, which will also be the crucial tool for us, is an essentially optimal inverse theorem for subset sums. %
Roughly speaking, the main result of \cite{conlon2023homogeneous} (for the precise statement, see Theorem \ref{thm:struc} below) says that for $A\subseteq [n]$ of size at least $n^{1/\beta}$ for some $\beta>1$ and $s$ slightly smaller than $|A|$, one can remove few elements of $A$ upon which $A$ is contained in a generalized arithmetic progression (GAP) $P = P(s; A)$ of bounded dimension and $\Sigma(A')$ contains a proper homogeneous translate of $(cs) P$ for an appropriate $A'\subseteq A$ of size $s$. On the other hand, $\Sigma(A')$ is clearly contained in $s P$. Here, a generalized arithmetic progression of dimension $d$ is a set of the form $P=\{x+\sum_{i=1}^{d}n_iq_i\,:\, 0 \le n_i \le w_i-1\}$ and $P$ is said to be homogeneous if $x$ is an integer linear combination of $q_1,\dots,q_d$. In particular, essentially, up to a constant $c$, we can encapsulate an arbitrary set $A$ by the much simpler GAP $P$ for which the structure of $\Sigma(A)$ is captured by the correct multiple of $P$.

In \cite{conlon2023homogeneous}, the structure theorem is used to show that, for $|A|$ exceeding $n^{\sqrt{2}-1+o(1)}$, if $\Sigma(A)$ does not contain a long homogeneous progression, then $A$ can be embedded in a small $2$-dimensional GAP. This structure is then used to find a one-dimensional fiber of the $2$-dimensional GAP which has too large an intersection with $A$ to be non-averaging.

Recalling Bosznay's construction, one can observe that the non-averaging set constructed indeed has a $2$-dimensional (or $3$-dimensional) structure. On the other hand, the argument of \cite{conlon2023homogeneous} does not capture tightly the structure of the construction. In fact, for obtaining an upper bound closer to the lower bound by Bosznay's construction, one must go beyond using only homogeneous progressions. 

We next turn to an informal sketch of the argument behind our proof of Theorem \ref{thm:main-0}. The crux of our argument shows that via sufficiently good understanding of the additive properties of a large non-averaging subset of $[n]$, such a set must resemble the projection of a point set in convex arrangement in a higher dimension. This roughly matches the key insight behind Bosznay's construction, and in particular, our argument mirrors the important structure behind the construction. 

Revisiting the argument of Erd\H{o}s and Straus, for a partition $A=\{a\}\cup A_1\cup A_2$ with $|A_1|\approx |A_2|$, we wish to show that $\Sigma(A_1-a) \cap \Sigma(a-A_2) \ne \emptyset$. We then apply the structure theorem, Theorem \ref{thm:struc}, to the sets $A$, $A_1-a$ and $a-A_2$. For an illustration, suppose for the moment that $P(s; A)$ has dimension $2$, and in particular, we have an embedding of $A$ into a box in $\mathbb{Z}^2$. We henceforth view $A$ as a subset of $\mathbb{Z}^2$ via this embedding. By careful preprocessing of the set $A$, we can assume that $P(s_1;A_1)$ and $P(s_2;A_2)$ have dimension $2$ and interact nicely with $P(s;A)$. Notice that $(cs_1)P(s_1;A_1)$ and $(cs_2)P(s_2; A_2)$, for $s_1,s_2$ sufficiently large, would intersect as long as their convex hulls intersect. Indeed, the key insight behind Bosznay's construction is that the convex hull of $A_1-a$ and $a-A_2$ are disjoint for all partitions $A=\{a\}\cup A_1\cup A_2$. On the other hand, requiring the convex hull of $A_1-a$ and $a-A_2$ to be disjoint for all partitions $A=\{a\}\cup A_1\cup A_2$ is equivalent to requiring that the points in $A$ are in convex position. In such situation, an upper bound on the size of $A$ can be easily derived (using e.g. Andrew's theorem on lattice polytopes \cite{Andrews}). 

However, there are several significant complications to the sketch above. In particular, applications of the structure theorem require removal of elements of $A$, and furthermore, the GAPs guaranteed to be contained in $\Sigma(A_1-a)$ or $\Sigma(a-A_2)$ are not $(cs_1)P(s_1;A_1)$ and $(cs_2)P(s_2; A_2)$ but certain translates of these GAPs. It turns out that, under these considerations, disjointness of $\Sigma(A_1-a)$ and $\Sigma(a-A_2)$ for all partitions $A=\{a\}\cup A_1\cup A_2$ translates to a noisy version of convexity of points in $A$, where for every point $a\in A$, there is a line through $a$ that is nearly a separating line in that at most $\delta |A|$ points of $A$ lie on one side of the line. We say that $A$ is in $\delta$-convex position if the above condition holds. 

Our strategy towards Theorem \ref{thm:main-0} is then a density increment argument. The key convex geometry lemma behind the density increment is that, given a discrete set $A$ inside a convex domain $\Omega$ in $\delta$-convex position, we can find a smaller convex domain $\Omega'$ for which the density of $A\cap \Omega'$ in $\Omega'$ increases significantly. 

Along the iteration, in applying the structure theorem, we do not have control over the dimension of the GAP $P(s;A)$, and hence along the proof, it is in fact more convenient for us to prove a more general result characterizing the size of the maximum subset of the $d$-dimensional box $[n]^d$. The definition of a non-averaging set generalizes readily to higher dimensions, i.e. a subset $A\subset \mathbb{Z}^d$ is \emph{non-averaging} if there is no element $a$ in $A$ which can be written as an average of a subset $A'$ not containing $a$. 

\begin{theorem}\label{thm:main}
    Let $A \subseteq B$ be a non-averaging set of a $d$-dimensional axis-aligned box $B$. Then $|A| \le |B|^{\alpha_d+o(1)}$ where $\alpha_d = (d-1)/(d+1)$ for $d\ge 2$ and $\alpha_1 = 1/4$. Here the $o(1)$ tends to 0 as $d$ is fixed and $n$ tends to infinity.
\end{theorem}

We note that the bound is tight for $d\ge 2$ (up to the $o(1)$ term in the exponent), by considering points arranged on a paraboloid in $d$ dimensions. The determination of the exponents $\alpha_d$ is obtained through various constraints arising in the density increment iteration. It is also here that one can pin down the dimensions supporting tight constructions for the non-averaging set problem in $1$-dimension. 

As illustrated through the sketch, our proof of Theorem \ref{thm:main-0} brings forth explicitly the role of convexity, and mirrors the structure of the construction giving the matching lower bound. Constructions based on convexity for sets avoiding certain linear configurations, which in particular include non-averaging sets, are common in additive combinatorics. On the other hand, this is the first case in which our structural understanding is sufficiently good to bring forth the underlying convex structure. We hope that some ideas from the present paper can be useful in broader contexts for similar problems. We also think that a more precise version of Theorem \ref{thm:main-0}, or stability and precise characterization of the structure of maximum non-averaging sets, would be interesting for further investigation.

$\qquad $

\noindent {\em Structure of the paper.} In Section \ref{section:convex_lemma}, we introduce and prove our main convex geometry lemma on sets in $\delta$-convex position. In Section \ref{section:structure_theorem}, we introduce the structure theorem, the preprocessing of the set $A$, and relate appropriate embeddings of a non-averaging set with a set in $\delta$-convex position. In Section \ref{section:non-averaging}, we then finish the proof of Theorem \ref{thm:main} and Theorem \ref{thm:main-0} by executing the density increment. 

$\qquad$

\noindent {\em Notations.} We use standard asymptotic notation. We write $a \ll b$ or $a=O(b)$ if $a \le C b$ for some constant $C>0$. A subscript $\ll_\beta$ or $O_\beta$ means that the implied constant $C$ depends on $\beta$. For clarity, we omit floor and ceiling signs where not essential.

\section{Convex geometry lemma}\label{section:convex_lemma}

We say that a finite set $A \subset \R^d$ is in {\em $\delta$-convex position} if for any $a \in A$ there is a half-space $H^+$ containing $a$ and such that $|A \cap H^+| \le \delta |A|$.

The goal of this section is to establish the following result.
\begin{lemma}\label{lem:convex1}
    For any $d \ge 1$ and $\varepsilon >0$ there are $\tau \in (0,1)$ and $\delta_0 >0$ such that for any $\delta \in (0,\delta_0)$ the following holds.
    Let $\Omega \subset \R^d$ be a convex body with non-empty interior and let $A \subset \Omega$ be a set in $\delta$-convex position. Assume that $|A|$ is sufficiently large in $\delta$. Then there exists some $\eta \in [\delta, \delta^\tau]$ and a convex subset $\Omega' \subset \Omega$ such that $\mathrm{Vol}_d(\Omega') \le \eta \mathrm{Vol}_d(\Omega)$ and $|A \cap \Omega'| \ge \eta^{\frac{d-1}{d+1}+\varepsilon}|A|$.
\end{lemma}

To see where the numerology comes from, consider the case when $A$ is uniformly distributed on the surface of a unit sphere $S^{d-1}$. Taking $\Omega'$ to be the convex hull of a spherical cap of radius $\sim r$ we see that $\mathrm{Vol}(\Omega') \sim r^{d+1}$ and $|\Omega' \cap A| \sim r^{d-1} |A|$. So Lemma \ref{lem:convex1} asserts that this is essentially the worst case. See \cite{Andrews, Barany} for related results.

Roughly speaking, our strategy for proving Lemma \ref{lem:convex1} is to first identify a dyadic scale of boxes in which we can pass to a relatively large subset of $A$ whose intersection with the boxes, if nonempty, has similar sizes. As long as the intersection of each box and $A$ is sufficiently large, by the assumption that $A$ is in $\delta$-convex position, we obtain that the boxes $B$ with nonempty intersection with $A$ are roughly in convex position, in that each box $B$ when enlarged by a constant factor must intersect with the convex hull $P$ of the boxes. Upon a random rotation and passing to a further subset, we can assume without loss of generality that points of $A$, the relevant boxes $B$ and their convex hull $P$ are represented as a function over a box $[-u,u]^{d-1}\subseteq \mathbb{R}^{d-1}$. Let $h$ denote the function representing $P$. %
We then show that we can find a further subbox $Q'$ of $Q$ of widths $u/m$ for which $h$ is well-approximated by a linear function $\ell$ (Lemma \ref{lem:convex-lin}). Furthermore, we can guarantee that $Q'$ contains the projection of many points in $A$. The domain $\Omega'$ is then taken to be the set of points $(v,h')$ where $v\in Q'$ and $h'$ is close to the linear approximation $\ell$ of $h$ over $Q'$. 

For a figurative representation of the proof plan, see Figure \ref{fig}. 

\begin{figure}
\tikzset{every picture/.style={line width=0.75pt}} %set default line width to 0.75pt        

\begin{tikzpicture}[x=0.75pt,y=0.75pt,yscale=-1,xscale=1]
%uncomment if require: \path (0,414); %set diagram left start at 0, and has height of 414

%Shape: Polygon Curved [id:ds5362119574857748] 
\draw  [color={rgb, 255:red, 208; green, 2; blue, 27 }  ,draw opacity=1 ] (224,164) .. controls (232.99,159.5) and (257.07,152.17) .. (285.95,145.75) .. controls (314.82,139.34) and (293.82,144.04) .. (297.85,143.23) .. controls (301.87,142.42) and (376.95,131.11) .. (398,139) .. controls (419.05,146.89) and (524.32,192.15) .. (543,238) .. controls (561.68,283.85) and (589.48,344.69) .. (590,364) .. controls (590.52,383.31) and (83,401) .. (63,371) .. controls (43,341) and (78.57,279.78) .. (114,238) .. controls (149.43,196.22) and (215.01,168.5) .. (224,164) -- cycle ;
%Shape: Square [id:dp47321814356622094] 
\draw   (68,313) -- (86,313) -- (86,331) -- (68,331) -- cycle ;
%Shape: Square [id:dp12783888082214612] 
\draw   (322,144) -- (340,144) -- (340,162) -- (322,162) -- cycle ;
%Shape: Square [id:dp5901339321487623] 
\draw   (293,148) -- (311,148) -- (311,166) -- (293,166) -- cycle ;
%Shape: Square [id:dp25809879052856954] 
\draw   (262,153) -- (280,153) -- (280,171) -- (262,171) -- cycle ;
%Shape: Square [id:dp6086320152744789] 
\draw   (224,164) -- (242,164) -- (242,182) -- (224,182) -- cycle ;
%Shape: Square [id:dp8299643467067404] 
\draw   (197,181) -- (215,181) -- (215,199) -- (197,199) -- cycle ;
%Shape: Square [id:dp07509659538698088] 
\draw   (165,197) -- (183,197) -- (183,215) -- (165,215) -- cycle ;
%Shape: Square [id:dp40278742712429216] 
\draw   (131,225) -- (149,225) -- (149,243) -- (131,243) -- cycle ;
%Shape: Square [id:dp7485534947139946] 
\draw   (108,254) -- (126,254) -- (126,272) -- (108,272) -- cycle ;
%Shape: Square [id:dp43664743352924673] 
\draw   (89,282) -- (107,282) -- (107,300) -- (89,300) -- cycle ;
%Shape: Square [id:dp5354166208142702] 
\draw   (564,340) -- (582,340) -- (582,358) -- (564,358) -- cycle ;
%Shape: Square [id:dp4394997778295734] 
\draw   (549,300) -- (567,300) -- (567,318) -- (549,318) -- cycle ;
%Shape: Square [id:dp9500152139074673] 
\draw   (514,221) -- (532,221) -- (532,239) -- (514,239) -- cycle ;
%Shape: Square [id:dp5669128123121449] 
\draw   (460,179) -- (478,179) -- (478,197) -- (460,197) -- cycle ;
%Shape: Square [id:dp19808663178094177] 
\draw   (429,164) -- (447,164) -- (447,182) -- (429,182) -- cycle ;
%Shape: Square [id:dp5690348734334609] 
\draw   (405,149) -- (423,149) -- (423,167) -- (405,167) -- cycle ;
%Shape: Square [id:dp21954739272006352] 
\draw   (380,139) -- (398,139) -- (398,157) -- (380,157) -- cycle ;
%Shape: Square [id:dp6196619518357532] 
\draw   (350,139) -- (368,139) -- (368,157) -- (350,157) -- cycle ;
%Straight Lines [id:da6366060008448311] 
\draw    (193,52) -- (194,409) ;
%Straight Lines [id:da152332347840633] 
\draw    (342,51) -- (343,408) ;
%Straight Lines [id:da43589258379824425] 
\draw    (483,52) -- (484,409) ;
%Straight Lines [id:da38208250329855176] 
\draw [color={rgb, 255:red, 144; green, 19; blue, 254 }  ,draw opacity=1 ][line width=2.25]    (193,175) -- (342,133) ;
%Straight Lines [id:da7121966288450421] 
\draw    (41,360) -- (639,361.99) ;
\draw [shift={(641,362)}, rotate = 180.19] [color={rgb, 255:red, 0; green, 0; blue, 0 }  ][line width=0.75]    (10.93,-3.29) .. controls (6.95,-1.4) and (3.31,-0.3) .. (0,0) .. controls (3.31,0.3) and (6.95,1.4) .. (10.93,3.29)   ;
%Straight Lines [id:da12686962334543794] 
\draw [color={rgb, 255:red, 126; green, 211; blue, 33 }  ,draw opacity=1 ][line width=2.25]    (192,211) -- (341,169) ;
%Straight Lines [id:da840173982070576] 
\draw [color={rgb, 255:red, 126; green, 211; blue, 33 }  ,draw opacity=1 ][line width=2.25]    (193,139) -- (342,97) ;

% Text Node
\draw (583,312.4) node [anchor=north west, draw=none, fill=none][inner sep=0.75pt]    {$g( y)$};
% Text Node
\draw (634,370.4) node [anchor=north west, draw=none, fill=none][inner sep=0.75pt]    {$y$};
% Text Node
\draw (303,114.4) node [anchor=north west, draw=none, fill=none][inner sep=0.75pt]    {$\ell ( y)$};
% Text Node
\draw (206,110.4) node [anchor=north west, draw=none, fill=none][inner sep=0.75pt]    {$\Omega '$};

\end{tikzpicture}
\caption{\label{fig}}
\end{figure}

Before proving Lemma \ref{lem:convex1}, we first derive several useful results around linear approximation of bounded convex functions. In particular, the next lemma shows that, given a convex function on a box in $(d-1)$-dimension, and a large collection of prescribed sub-boxes, the function is well-approximated by a linear function on at least one sub-box.

\begin{lemma}\label{lem:convex-lin}
	Let $m$ and $d$ be positive integers and let $c>2d/m$. Let $H : (-c,1+c)^{d-1} \to [0,1]$ be a convex function. For $v\in [m]^{d-1}$, let $Q_v$ be the cube $v/m + [0,1/m]^{d-1}$. Consider an arbitrary subset $I \subset [m]^{d-1}$. Then there exists $v\in I$ and a linear function $L_v$ on $Q_v$ such that 
\[
	\sup_{x \in Q_v} |H(x)-L_v(x)| \le \frac{4d^4m^{d-3}}{c|I|}.
\] 
\end{lemma}

To illustrate the numerology in this lemma, let us take $I = [m]^{d-1}$ and take $H$ to be a $C^2$-smooth function with bounded second derivatives. Then for $v \in [0,1]^{d-1}$ we have Taylor approximation
\[
H(v+x) = H(v)+\nabla_v H (x) + O(\|x\|_\infty^2)
\]
so if we have $\|x\|_\infty \le 1/m$ then we can take the linear function $L_v = H(v)+\nabla_v H$ and obtain $|H(v+x) - L_v(x)| = O(m^{-2}) = O(\frac{m^{d-3}}{|I|})$, matching the bound stated in the lemma. Since a convex function is differentiable almost everywhere this suggests that Taylor expansion gives a good approximation in neighborhoods of most points $v \in [0,1]^{d-1}$. Lemma \ref{lem:convex-lin} is a discrete version of this continuous statement and it moreover gives us a precise control on the set of `exceptional' points $v$ where Taylor approximation fails.

%Lemma \ref{lem:convex-lin} gives us precise quantitative control on the size of the set of points $v \in [0,1]^{d-1}$ where Taylor approximation fails by a given amount. 

We will use the following basic claim. We let $e_1,\ldots, e_{d-1}$ denote the standard basis vectors in $\R^{d-1}$. 
\begin{claim}\label{claim:incr-der}
    Let $m$ and $d$ be positive integers and let $c$ be such that $c > d/m$. 
	Let $H : (-c,1+c)^{d-1} \to [0,1]$ be convex and differentiable. For $v\in [m]^{d-1}$, assume that there is no function $L_v$ on $Q_v$ such that $\sup_{x\in Q_v} |H(x)-L_v(x)|\le \delta$. Then there exists $i\in [d-1]$ such that $\partial_i H(v/m+(d/m)e_i)\ge \partial_i H(v/m) + \frac{\delta m}{d^2}$. 
\end{claim}
\begin{proof}
	Let $u$ be the gradient of $H$ at $v$. Let $L_v$ be the linear function agreeing with $H$ at $v$ and with gradient $u$. By assumption, there exists $y\in Q_v$ such that $H(y) \ge L_v(y) + \delta$. Writing $y = v/m + \sum_{i=1}^{d-1}a_i e_i$ for the standard basis $e_1,\dots,e_{d-1}$ and $s=\sum_{i=1}^{d-1}a_i \in (0,d/m]$, we have $y = \sum_{i=1}^{d-1} \frac{a_i}{s} (v/m+se_i)$ so by convexity,
\[
	\sum_{i=1}^{d-1} \frac{a_i}{s} H(v/m+se_i) \ge H(y) \ge L_v(y)+\delta = \sum_{i=1}^{d-1} \frac{a_i}{s} L(v/m+se_i) + \delta.
\]
As such, there exists $i\in [d-1]$ such that $H(v/m+se_i) \ge L(v/m+se_i) + \delta/d$. Because of this and $H(v/m)= L(v/m)$ and $s \le d/m$, there exists $s' \in [0, d/m]$ such that $\partial_i H(v/m+ s' e_i) \ge \partial_i L(v/m+s' e_i) +\frac{\delta m}{d^2}$. Since $H$ is convex and $L$ is linear and $\partial_i L(v/m+s'e_i) =\partial_i L(v/m) = \partial_i H(v/m)$, we conclude that 
\[
\partial_i H(v/m+(d/m)e_i) \ge \partial_i H(v/m + s' e_i) \ge \partial_i L(v/m + s' e_i) + \frac{\delta m}{d^2} = \partial_i H(v/m) + \frac{\delta m}{d^2} 
\]
as desired.
\end{proof}

\begin{proof}[Proof of Lemma \ref{lem:convex-lin}]
    Note that, by standard convolution with smooth mollifiers, there exist smooth convex functions whose $L_\infty$ distance to $H$ on $(-c/2,1+c/2)$ are arbitrarily small. Thus, in proving Lemma \ref{lem:convex-lin}, we can assume without loss of generality that $H$ is everywhere differentiable in $(-c/2,1+c/2)^{d-1}$. 
    
	Assume that for all $v\in I$, there is no linear function $L_v$ on $Q_v$ such that $|H(x)-L_v(x)| \le \delta = \frac{4d^4 m^{d-3}}{c |I|}$ for all $x \in Q_v$. Then by Claim \ref{claim:incr-der}, for each $v\in I$, there is $i_v\in [d-1]$ such that $\partial_i H(v/m+(d/m)e_i)\ge \partial_i H(v/m) + \frac{\delta m}{d^2}$. 
	
	Let $\pi_i$ be the projection $\pi_i(v_1,\dots,v_{d-1}) = (v_1, \ldots, v_{i-1}, v_{i+1}, \ldots, v_{d-1})$. By the pigeon-hole principle, there exists $i\in [d-1]$ and $a\in [m]^{d-2}$ such that $\pi_i^{-1}(a)$ contains $\ell \ge |I|/(d^2m^{d-2})$ points $v\in I$ for which $i_v=i$, and furthermore every two such points are separated by distance at least $d$. Order the points $v\in I\cap \pi_i^{-1}(a)$ by $v_1,\dots,v_\ell$, we then have $\partial_i H(v_{j+1}/m) \ge \partial_i H(v_{j}/m) + \frac{\delta m}{d^2}$ and so 
     \[
     \partial_i H(v_{\ell}/m) \ge \partial_i H(v_{1}/m)  + \frac{\ell \delta m}{d^2}.
     \]
     So we get that $\partial_i H(v_{\ell}/m) \ge \frac{\ell \delta m}{2d^2}$ or $\partial_i H(v_{1}/m) \le -\frac{\ell \delta m}{2d^2}$. In either case, by convexity of $H$ on $(-c/2,1+c/2)^{d-1}$, we then have 
    \[
        1 \ge \sup H - \inf H \ge \frac{c}{2}\frac{\ell \delta m}{2d^2}.
    \]
    This contradicts our choice of $\delta$. 
\end{proof}

We are now ready to prove Lemma \ref{lem:convex1}. 
\begin{proof}[Proof of Lemma \ref{lem:convex1}]
Any convex body $\Omega \subset \R^d$ can be covered by a rectangular box $B$ whose volume exceeds the volume of $\Omega$ by a constant factor depending on $d$. So by replacing $\Omega$ with $B$ and making an affine change of the coordinates we may assume that $\Omega = [-1,1]^d$. 

Without loss of generality, we may assume that $\varepsilon \le \frac{1}{d+1}$. We then take $\tau = 0.1 \varepsilon$ and let $\delta_0$ be sufficiently small depending on $\varepsilon$ and $d$. 

For $r \in 2^{-\N}$, a dyadic $r$-box $B \subset \R^d$ is a set of points of the form
$$
B = [x_1, x_1+r) \times [x_2, x_2+r) \times \ldots \times [x_d, x_d+r),
$$
for some $x_i \in r\Z$.

Let $r = C\delta^{1/d}$ be a dyadic number for some suitable constant $C \in (8, 16]$. For $\mu \in 2^{-\N}$ define $\mathcal B_\mu$ to be the set of dyadic $r$-boxes $B$ such that $|B \cap A| \in [\mu |A|, 2\mu|A|)$. Note that the total number of dyadic $r$-boxes $B \subset \Omega$ is $(2/r)^d$, so for $c = 2^{-d-2}$ we have
\[
\sum_{c r^d \le \mu \le 1} \sum_{B \in \mathcal B_\mu} |A \cap B| \ge (1-2^{d+1}c) |A| \ge |A|/2.
\]
Thus by the pigeon-hole principle there exists dyadic scale $\mu \ge c r^d$ such that 
\begin{equation}\label{eq:I}
\sum_{B \in \mathcal B_\mu} |A \cap B| \ge \frac{|A|/2}{\log_2 (1/cr^d)} \ge \frac{|A|}{4\log(1/\delta)},    
\end{equation}
where the last inequality holds for $\delta\le \delta_0$. 
Label $\mathcal B_\mu = \{B_i\}_{i\in I}$ and let $A' = \bigcup_{i\in I} (A\cap B_i)$. Note that by the choice of $C$ and $c$, we have $|A\cap B_i| \ge \mu |A| \ge c r^d |A| = 2^{-d-2} C^d \delta |A| > \delta |A|$.

Observe that the collection of boxes $B_i$ is in convex position in the following sense. For a suitable constant $C' = 4\sqrt{d}$ let $\tilde B_i$ be the $C'r$-box with the same center as $B_i$, then 
\begin{equation}\label{eq:conv}
\tilde B_i \not \subset \operatorname{conv}\left(\bigcup_{j \in I\setminus \{i\}} B_j\right)    
\end{equation}
for any $i \in I$. Indeed, since the set $A$ is assumed to be in $\delta$-convex position and $|A\cap B_j| > \delta |A|$ for any $j\in I$, there exists a hyperplane $H$ partitioning $\mathbb{R}^d = H^- \cup H^+$ such that $H\cap B_i \neq \emptyset$ and $B_j \not \subset H^+$ for any $j \in I\setminus \{i\}$. If we let $H' \subset H^+$ be a hyperplane parallel to $H$ at a distance $ 2\sqrt{d} r$ from $H$, then $\tilde B_i \cap H' \neq \emptyset$ and $B_j \cap H'^{+} = \emptyset$ for every $j\in I\setminus \{i\}$, yielding (\ref{eq:conv}).

Let $P = \operatorname{conv}(\bigcup_{i\in I} B_i)$, this is a convex polytope inside $[-1,1]^d$. For each $i\in I$, let $x_i \in \partial P \cap \tilde B_i$ be an arbitrary point on the intersection of the boundary of $P$ and $\tilde B_i$ (by (\ref{eq:conv}) such a point exists). 
% Note that
% \[
% A' \subset \bigcup_{i\in I} B_i \subset P \subset P + [-C'r, C'r]^d.
% \]
Let $w$ denote the width of $P$ and note that since $P \subset [-1,1]^d$, we have $\operatorname{Vol}_d(P) \le C' w$ for some constant $C'$ depending only on $d$. 
%note that the width of the body $P' = P + [-C'r, C'r]^d$ is at most $2\sqrt{d}C'r+w$. Since $P' \subset [-2,2]^{d}$, we conclude that $\mathrm{Vol}_d(P') \le C'' (r+w)$ for 
%a constant $C''$ depending only on $d$.
If $w \le \delta^{\tau} / 2C'$ then the set $\Omega' = P$ satisfies the conclusion of the lemma with $\eta = C'w$. Indeed, from (\ref{eq:I}) we easily get $|A\cap P| \ge \frac{|A|}{4\log(1/\delta)} \ge \eta^{\frac{d-1}{d+1}}|A|$ provided that $\delta \le \delta_0$ is small enough.  
Hence, we may assume that $w \ge \delta^{\tau}/2C'$ holds.

Under this assumption, we conclude that $P$ contains a ball of radius $\delta^\tau/4C'$. Translate $P$ so that it contains the ball of radius $\delta^\tau/4C'$ with center at the origin. Note that after the translation, $P \subset [-2,2]^d$ holds. Let $g \in \mathrm{SO}(d)$ be a uniformly random orthogonal map and denote $\pi: \R^{d} \rightarrow \R^{d-1}$ the projection onto the first $d-1$ coordinates. For a suitable constant $c' = \frac{1}{8C' \sqrt{d}}>0$, let $u = c' \delta^{\tau}$ and define $I(g)$ to be the set of indices $i \in I$ such that 
\begin{equation}\label{eq:projection}
\pi(g(x_i)) \in [-u,u]^{d-1} \text{ and }g(x_i)_{d} > 0
\end{equation}
Note that each $x_i$ is a vector of length at most $2\sqrt{d}$ and since, for fixed $i$, $g(x_i)$ is a uniformly random vector of length $\|x_i\|_2$, the probability of (\ref{eq:projection}) is at least $c'' u^{d-1}$ for each fixed $i$ for some constant $c''>0$ depending on $d$. By linearity of expectation, there is a choice of $g$ such that $|I(g)| \ge c'' u^{d-1} |I|$. Fix this choice of $g$ for the rest of the proof. 

For $y \in [-2u, 2u]^{d-1}$ let $h(y) > 0$ be largest such that $(y, h(y)) \in g(P)$. Since $P$ contains the ball of radius $\delta^\tau/4C' = 2\sqrt{d}u$ around the origin, we can guarantee $[-2u, 2u]^{d-1} \subset g(P)$ and hence $h$ is well-defined on the box $[-2u, 2u]^{d-1}$. Using that $P$ is convex, we get that $h$ is a concave function on the box $[-2u, 2u]^{d-1}$, i.e. 
\[
h(\alpha y + (1-\alpha) y') \ge \alpha h(y) + (1-\alpha) h(y')
\]
for any $y, y' \in [-2u,2u]^{d-1}$ and $\alpha \in [0,1]$. Since $g(P) \subset [-2\sqrt{d}, 2\sqrt{d}]^d$, we have $h(y)\le 2\sqrt{d}$.
Let $y_i \in [-u, u]^{d-1}$, $i\in I(g)$, be the unique vector such that $g(x_i) = (y_i, h(y_i))$ and consider the set $Y = \{y_i, ~i\in I(g)\}$. Fix an integer $m = \lceil\delta^{-\frac{0.3}{d+1}}\rceil$ and for $v = (v_1, \ldots, v_{d-1}) \in [-m, m)^{d-1}\cap \Z^{d-1}$ let 
\[
Q_v = \left[\frac{u}{m} v_1, \frac{u}{m} (v_1+1)\right) \times \ldots \times \left[\frac{u}{m} v_{d-1}, \frac{u}{m} (v_{d-1}+1)\right) \subset [-u,u]^{d-1}.
\]
For dyadic $K> 0$ define
\[
J_K = \left\{v \in [-m,m)^{d-1} \cap \Z^{d-1}:~ |Q_v \cap Y| \in \left[\frac{K |Y|}{m^{d-1}}, \frac{2K|Y|}{m^{d-1}}\right) \right\}.
\]

Note that we have $|Q_v \cap Y| \le |Y| \le |I(g)| \le \delta^{-1}$, so $K$ can take up to $\log_2(1/\delta)$ dyadic values. Thus, there exists $K$ such that
\begin{equation}\label{eq:dyadic-thing}
\frac{|Y|}{\log_2(1/\delta)} \le \sum_{v \in J_K} |Q_v \cap Y| \le |J_K| \frac{2K|Y|}{m^{d-1}}    
\end{equation}
and so $|J_K| \ge \frac{1}{2\log(1/\delta)}\frac{m^{d-1}}{K}$. Note that since $|J_K| \le (2m)^{d-1}$, (\ref{eq:dyadic-thing}) implies that $K \ge \frac{1}{2^{d} \log_2(1/\delta)}$.

Let us suppose that $K \le m^{\frac{d-1}{d+1}}$ holds, we will deal with the (easy) case when $K$ is large at the end of the proof.
Let $\mathbf{1} \in \R^{d-1}$ be the vector whose coordinates are all $1$ and consider the function $H(x) = -\frac{1}{2\sqrt{d}}h(u(2 x-\mathbf{1})) + 1$ defined on $[-1/2, 3/2]^{d-1}$ and taking values in $[0,1]$. By Lemma \ref{lem:convex-lin} applied to the function $H$ and the set $J_K$, we can then find $v\in J_K$ and a linear function $L$ such that $|H(x) - L(x)| \le \frac{8d^4 (2m)^{d-3}}{|J_K|} = 2^d d^4 \frac{m^{d-3}}{|J_K|}$ for $x \in (2u)^{-1}Q_v + {\bf 1}$. So undoing the change of coordinates gives a linear function $\ell$ for which 
\[
|h(v’)-\ell(v’)|\le \frac{2^{d+1}d^{9/2}m^{d-3}}{|J_K|} \le  \frac{C_{1} K\log (1/\delta)}{m^2} 
\] 
for all $v’\in Q_v$, for some constant $C_{1}$ depending only on $d$. In particular, note that using the restriction $K \le m^{\frac{d-1}{d+1}} \ll m$ for sufficiently small $\delta$, we get $|h(v') -\ell(v')| \le 0.1 m^{-1} \le 0.1$. 

Write $\ell(y) = \sum a_i y_i -b$. We claim that $|a_i| \le C_2 u^{-1}$ for a suitable constant $C_2 = 3\sqrt{d}$ depending only on $d$. Indeed, if we have $|a_i| \ge C_2 u^{-1}$ for some $i$ then for $z = \operatorname{sign}(a_i) u e_i \in [-u, u]^{d-1}$ we get $\ell(z) - \ell(0) = |a_i| u \ge C_2$.
%otherwise for some $z \in [-2u,2u]^{d-1}$ we have $\ell(z)-\ell(0) \ge  C_2/2$: if we have $|a_i| \ge C_2 u^{-1}$ for some $i$ then one of $z = v\pm ue_i$ satisfies $\ell(z) = |a_i| u + \ell(v) \ge C_2 -0.1 \ge C_2/2$, if $b \ge C_2$ then $\ell(0) \ge C_2$ and if $b\le -C_2$ then $\ell(2v) = 2\ell(v) - \ell(0) \ge -0.2 + C_2 \ge C_2/2$.
Now if we take $v'$ so that $v', v'+\frac{z}{m} \in Q_v$, then by concavity
\begin{align*}
h(v'-z) &\le (m+1) h(v') - m h(v' + z/m) \le \\ &\le h(v') + m (\ell(v') - \ell(v'+z/m)) + 0.2 \le \\
&\le 
h(v') - C_2 + 0.2.    
\end{align*}
Since $C_2 = 3\sqrt{d}$, this contradicts $h(y) \in [0, 2\sqrt{d}]$.

We now proceed to define the desired convex set $\Omega'$. Let
\begin{equation}\label{eq:omega}
\Omega' = \{ (v', h') \in \R^d:~ v' \in Q_v, ~ |h' - \ell(v')| \le \frac{C_{1} K\log (1/\delta)}{m^2} \}.    
\end{equation}
It is clear that $\Omega'$ is a convex body of volume at most $2C_{1} K(\log (1/\delta)) u^{d-1} m^{-d-1}$ and that it contains all points $g(x_i)=(y_i, h(y_i))$ for $y_i \in Y \cap Q_v$. So since we chose $v \in J_K$, we have
\[
|\Omega'\cap \{g(x_i)\}| \ge |Y\cap Q_v| \ge \frac{K|Y|}{ m^{d-1}}.
\]
Recalling our construction, each $x_i$ corresponds to a $C'r$-box $\tilde B_i$ such that $x_i\in \partial P \cap \tilde B_i$ and $\tilde B_i$ contains an $r$-box $B_i$ such that $|A\cap B_i| \in [\mu |A|, 2\mu|A|]$. If we take $S$ to be a $(2\sqrt{d} C') r$-ball centered at the origin, then we have $B_i \subset \tilde B_i \subset S+x_i $ and
\[
\bigcup_{i: g(x_i) \in \Omega'} g(B_i) \subset \Omega' + S
\]
and so we get
\[
|A \cap g^{-1}(\Omega' + S)| \ge \mu |A| |\Omega' \cap \{g(x_i)\}| \ge \mu |A| \frac{K|Y|}{ m^{d-1}}.
\]
Recall that $Y$ is the set of points $y_i$ such that $g(x_i) = (y_i, h(y_i) )$ and $i \in I(g)$. So $|Y|= |I(g)|$ and we chose $g$ so that $|I(g)| \ge c''u^{d-1} |I|$ holds. In turn, by (\ref{eq:I}), we have $|I| \ge \frac{1}{8\log (1/\delta) \mu}$. Putting these estimates together, we obtain
\begin{equation}\label{eq:A}
|A \cap g^{-1}(\Omega' + S)| \ge \frac{c''u^{d-1}K|A|}{8 m^{d-1} \log(1/\delta)}.    
\end{equation}
It remains to upper bound the volume of the body $\Omega''=g^{-1}(\Omega' + S)$. Since $g \in \mathrm{SO}(d)$, it is the same as the volume of $\Omega' + S$. Using (\ref{eq:omega}) and the bound on the coefficients of $\ell$, one can check that the width of $\Omega'$ is at least $\frac{C_{1}K(\log (1/\delta)) u}{2C_2\sqrt{d} m^2}$. Note that we chose $u = c'\delta^{\tau} \ge c' \delta^{\frac{0.1}{d+1}}$ since $\varepsilon \le \frac{1}{d+1}$ by assumption, $m = \lceil\delta^{-\frac{0.3}{d+1}}\rceil$ and $r = 4\delta^{1/d}$, so we obtain $u / m^2 \ge c' \delta^{\frac{0.7}{d+1}} > r$. 
Thus, $\Omega'+S$ has the same volume as $\Omega'$ up to a constant factor depending only on $d$. We conclude that 
\[
\mathrm{Vol}_d(\Omega'') = \mathrm{Vol}_d(\Omega' + S) \le \frac{C''(K\log (1/\delta)) u^{d-1}}{m^{d+1}}
\]
for some $C''$ depending on $d$. 
In the conclusion of Lemma \ref{lem:convex1}, let us define $\eta = \frac{C'' (K\log (1/\delta)) u^{d-1}}{m^{d+1}}$. We need to verify that 
\[
|A\cap \Omega''| \ge \eta^{\frac{d-1}{d+1} +\varepsilon} |A|
\]
holds. By (\ref{eq:A}) we get
\begin{align*}\label{eq:req-m}
\frac{|A\cap \Omega''|}{|A|} \eta^{-\frac{d-1}{d+1} -\varepsilon} &\ge \frac{c''u^{d-1} K}{ 8m^{d-1} \log (1/\delta)} \eta^{-\frac{d-1}{d+1} -\varepsilon}  \\ &\ge \frac{c''u^{d-1} K}{ 8m^{d-1} \log (1/\delta)}\left(\frac{C'' (K\log (1/\delta)) u^{d-1}}{m^{d+1}}\right)^{-\frac{d-1}{d+1}-\varepsilon} \\
&\ge m^{\varepsilon(d+1)} u^{\frac{2(d-1)}{d+1}} K^{\frac{2}{d+1}-\varepsilon} \ge m^{\varepsilon(d+1)} u^2 
\end{align*}
where we took sufficiently small $\delta$ and used the term $u^{-(d-1)\varepsilon}$ to consume constants and logarithms and at the end we used the lower bound $K \ge \frac{1}{2^d \log_2(1/\delta)}$. So using $m = \lceil\delta^{-\frac{0.3}{d+1}}\rceil$ and $u = c' \delta^{0.1\varepsilon}$, we get
\[
\frac{|A\cap \Omega''|}{|A|}\ge m^{\varepsilon (d+1)} u^2  \ge \delta^{-0.3 \varepsilon} c'^2\delta^{0.2\varepsilon} > 1,
\]
for sufficiently small $\delta$. This completes the proof of Lemma \ref{lem:convex1} in the case when $K \le m^{\frac{d-1}{d+1}}$.

% Recall that $K\ge c$. 
% Choosing $\tau = 0.1 \min\{\varepsilon, \frac{1}{d+1}\}$, and recalling that $u=c'\delta^\tau$ and $r=C\delta^{1/d}$ for suitable constants $c',C$ depending only on $d$, it is easy to check that, for $\delta$ sufficiently small, we can find an integer $m\in \left[u^{- \frac{2}{\varepsilon (d+1)}}, (u/r)^{1/2}\right]$ to ensure (\ref{eq:req-m}). 
% This completes the proof of Lemma \ref{lem:convex1}.

Let us now deal with the case $K \ge m^{\frac{d-1}{d+1}}$. In this case we do not use Lemma \ref{lem:convex-lin} and construct the convex set $\Omega'$ directly. 
Fix an arbitrary $v \in J_K$. Since $h: [-2u, 2u]^{d-1}\rightarrow [0, 2\sqrt{d}]$ is convex, we have $|\partial_i h(v')|\le 2\sqrt{d} u^{-1}$ for $v' \in [-u, u]^{d-1}$. Thus, $h(Q_v)$ is contained in an interval of length at most $\frac{2d^{3/2}}{m}$. Let $\Omega' = Q_v \times h(Q_v)$. Each $y \in Y\cap Q_v$ gives rise to a unique $i\in I(g)$ such that $\pi(g(x_i)) = y$ and $g(x_i) \in \Omega'$. As previously, we have the inclusion $B_i \subset x_i + S$ and so for $\Omega'' = g^{-1}(\Omega' + S)$ we get
\[
|A\cap \Omega''| =|A \cap g^{-1}(\Omega' + S)| \ge \sum_{i: g(x_i) \in \Omega'} |A \cap B_i| \ge \mu|A| |Y \cap Q_v| \ge \mu |A| \frac{K|Y|}{m^{d-1}}.
\]
Since $r \ll u$, the volume of $\Omega''$ is the same as the volume of $\Omega'$, up to a constant depending on $d$. So we can take $\eta = \frac{\operatorname{Vol}(\Omega'')}{\operatorname{Vol}(\Omega)} \sim u^{d-1} m^{-d}$ and compute by (\ref{eq:A})
\begin{align*}
\frac{|A\cap \Omega''|}{|A|} \eta^{-\frac{d-1}{d+1}-\varepsilon} \ge \frac{c''u^{d-1}K}{8 m^{d-1} \log(1/\delta)} \eta^{-\frac{d-1}{d+1}-\varepsilon} \ge \\
\ge m^{- \frac{d-1}{d+1} + \varepsilon (d+1)} u^{2 \frac{d-1}{d+1} } K \ge m^{\varepsilon (d+1)} u^2 > 1
\end{align*}
which completes the proof.
\end{proof}

\section{Structure theorem}\label{section:structure_theorem}

We recall several useful definitions from \cite{conlon2023homogeneous}, generalized to subsets of $\mathbb{Z}^{\ell}$. 

\begin{definition}

    A {\em generalized arithmetic progression} or GAP, for short, in $\Z^\ell$ is a set of the form $Q = \{x + \sum_{i=1}^d n_i q_i: ~ 0\le n_i \le w_i-1\}$ where $x, q_1, \ldots, q_d \in \Z^\ell$ and $d, w_1, \ldots, w_d$ are positive integers. Note that all these elements are part of the data of a GAP $Q$.

    We say $d$ is the {\em dimension} of $Q$ and define the {\em volume} of $Q$ as $\operatorname{Vol}(Q) = \prod_{i=1}^d w_i$.

    We say $Q$ is {\em proper} if $|Q| = \operatorname{Vol}(Q)$, that is, all vectors $x + \sum_{i=1}^d n_i q_i: ~ 0\le n_i \le w_i-1$ in the representation of $Q$ are pairwise distinct.

    We say $Q$ is {\em homogeneous} if $x$ is an integer linear combination of $q_1, \ldots, q_d$.
\end{definition}

A $d$-dimensional homogeneous GAP $Q$ can be written in the form $Q= \{\sum_{i=1}^d n_i q_i, ~ a_i \le n_i \le b_i \}$ for some real numbers $a_i < b_i$.
Given a choice of intervals $[a_i, b_i]$ defining a homogeneous GAP $Q$ and a positive real number $c$ we write $c Q = \{\sum_{i=1}^d n_i q_i: ~ca_i \le n_i \le cb_i \}$. Note that if $c$ and $a_i, b_i$ are all integers then $c Q$ coincides with the $c$-fold sumset of $Q$.

Given a proper GAP $Q$ of dimension $d$, we denote by $\phi_Q: Q \to \mathbb{Z}^d$ its identification map, i.e.
\[
\phi_Q(x + \sum_{i=1}^d n_i q_i) = (n_1, \ldots, n_d).
\]

The proof of \cite{conlon2023homogeneous}, in verbatim, gives the same structural theorem in higher dimension. For the reader's convenience, we give a self-contained derivation of Theorem \ref{thm:struc} assuming the main result of \cite{conlon2023homogeneous} in Appendix \ref{section:appendix}.
 
\begin{theorem}\label{thm:struc}
For any $\ell$, $\beta > 1$, and $0 < \eta < 1$, there are positive constants $c$ and $d$ such that the
following holds. Let $A$ be a subset of an axis-aligned box $B\subseteq \mathbb{Z}^\ell$ of size $m$ with $|B| \le m^\beta$ and let $s \in [m^\eta, cm/ \log m]$. Then there exists a subset $\hat A$ of $A$ of size at least $m-c^{-1} s \log m$ and an integer $d'\le d$ such that the following holds. There is a $d'$-dimensional GAP $P$ containing $\hat{A} \cup \{0\}$ and a subset $A'$ of $\hat A$ of size at most $s$ such that $\Sigma(A')$ contains a homogeneous translate of $csP$. Furthermore, $csP$ is proper.
\end{theorem}

In fact, for our purposes it will be sufficient to apply this result with the specific value $s = s(A) := |A|/ (\log |A|)^2$. For the readers convenience, we state a corollary of Theorem \ref{thm:struc} using this value.

\begin{cor}\label{cor:struc}
    For any $\ell, \beta>1$ there exists constants $c$ and $d$ such that the following holds. Let $A$ be a subset of an axis-aligned box $B\subseteq \mathbb{Z}^\ell$ of size $m$ with $|B| \le m^\beta$. Then there exists a subset $\hat A$ of $A$ of size at least $m-c^{-1}m/\log m$ and an integer $d'\le d$ such that the following holds. There is a $d'$-dimensional GAP $P$ containing $\hat{A} \cup \{0\}$ and a subset $A'$ of $\hat A$ of size at most $s(A)=m/\log^2 m$ such that $\Sigma(A')$ contains a homogeneous translate of $csP$. Furthermore, we can assume that $P$ is symmetric and $csP$ is proper.
\end{cor}

We say that $A$ is a $(\ell,\beta)$-set if $A\subseteq \mathbb{Z}^{\ell}$ and there exists a box $B \subseteq \mathbb{Z}^{\ell}$ such that $|B|\le |A|^{\beta}$. For a $(\ell,\beta)$-set $A$, define the {\em subset sum dimension} $d(A)$ of $A$ to be a number $d(A) = d'$ such that there exists a $d'$-dimensional GAP $P$ and subsets $A' \subset \hat A \subset A$ of size at most $s=s(A)$ and at least $|A|-c^{-1} s\log |A|$, respectively, so that $\hat A\cup \{0\} \subset P$ and $\Sigma(A')$ contains a homogeneous translate of $csP$ and $csP$ is proper. We denote the GAP $P$ containing $\hat{A} \cup \{0\}$ by $P(A)$. Given the GAP $P=P(A)$, we denote its identification map $\phi_P: \hat A \cup \{0\} \to \mathbb{Z}^{d'}$. 

Note that the definition of $d(A)$ implicitly depends on the particular choice of parameters $\ell$ and $\beta$. This is not going to be important in applications since these parameters will always stay bounded, so we suppress them from the notation.

%the conclusion of the theorem holds with the GAP $P$ and some $\hat A \subset A$ of size at least $|A|-c^{-1} s\log |A|$.   

\subsection{Reduction}

Observe that if $A$ is a non-averaging set, applying Corollary \ref{cor:struc} and considering the GAP $P = P(A)$ together with its identification map $\phi_P : \hat A \cup \{0\} \to \mathbb{Z}^{d'}$, we have that $\phi_P(\hat A \cup \{0\})$ is a non-averaging subset of the box $\phi_P(P)\subset \mathbb{Z}^{d'}$. 

We record the following simple observation.
\begin{lemma}\label{lem:shrink}
    Let $A$ be a subset of a $d$-dimensional GAP $P$. Assume that the subset sum dimension of $A$ is $d_1$. Then we have $|P(A)| \ll_d |P| s(A)^{-(d_1-d)}$.   
    %In particular, if $A \subseteq \mathbb{Z}^{\ell}$ is a subset of a box $B$ with $|B| \le |A|^{\beta}$, then $|P(A)| \ll_{\ell} |B| s(A)^{-(d_1-\ell)}$. 
\end{lemma}

\begin{proof}
    Let $P_1 = P(A)$ be the GAP given by Corollary \ref{cor:struc} and let $A' \subset \hat A \subset A$ be the corresponding sets. We have, for $s=s(A)$, 
    \[
        |c s P_1| \gg s^{d_1}|P_1|,
    \]
    since $cs P_1$ is proper and on the other hand,
    \[
        |cs P_1| \le |\Sigma (A')| \le |s P| \le s^d |P|.
    \]
    Hence, 
    \[
        |P_1|  \ll s^{-(d_1-d)} |P|.
    \]
\end{proof}
%For each set $A$, we define $s(A) = |A|/(\log |A|)^{2}$. 

For a similar result when $d_1 \le d$, we will make use of a (qualitative) discrete John's theorem \cite[Lemma 3.36]{TV}. Given a collection of vectors $\mathbf{v} = (v_1,\dots,v_d)$ and a tuple of integers $\mathbf{N} = (N_1,\dots,N_d)$, we define $P(\mathbf{v}, \mathbf{N}) = \{\sum_{i=1}^{d}n_iv_i: n_i \in [-N_i, N_i]\}$. 

\begin{lemma}\label{lem:john}
    Let $B$ be a symmetric convex subset of $\mathbb{R}^d$. Then there is an integer basis $\mathbf{v} = (v_1,\dots,v_d)$ and integers $\mathbf{N} = (N_1,\dots,N_d)$ such that $P(\mathbf{v}; c_d\mathbf{N}) \subseteq B\cap \mathbb{Z}^d \subseteq P(\mathbf{v}; \mathbf{N})$. %If $B$ is symmetric, we can take $x=x'=0$. Here $c_d>0$ is a constant depending on $d$. 
\end{lemma}

%\blue{TODO: reference}

\begin{lemma}\label{lem:decrease-dim}
    Consider a symmetric axis-aligned box $B$ in $\mathbb{Z}^d$. Let $A$ be a subset of $B$ with subset sum dimension $d' \le d$. Then we have $|P(A)| \ll_d |B|$.
\end{lemma}
\begin{proof}
    Let $P = P(A)$ and let $H$ be the subspace spanned by $P$. Note that the dimension of $H$ is at most $d'$. By Lemma \ref{lem:john}, $H\cap B$ is contained in a symmetric GAP $B' = P(\mathbf{v}, \mathbf{N})$ of size $O(|B|)$.

    By Corollary \ref{cor:struc}, we have a subset $A' \subset A$ so that
    \[
    x + cs P \subset \Sigma(A') \subset s B',
    \] 
    where $s=s(A)$, and since $cs P$ is proper and $d'$-dimensional,
    \[
    s^{d'} |P| \ll |cs P| \le |s B'| \ll s^{\dim H}|B|, 
    \]
    which implies $|P|\ll_d |B|$.
\end{proof}

We next introduce the notion of \emph{$(\delta,\gamma)$-irreducibility}, and show that in the argument, every non-averaging $(\ell,\beta)$-set $A$ can be replaced by a large $(\delta,\gamma)$-irreducible set. 

\begin{definition}
    We say that a $(\ell,\beta)$-set $A$ with $P=P(A)$ is $(\delta,\gamma)$-irreducible if for any subset $A'$ of $\phi_P(A)$ of size at least $\delta |A|$ and for any $x \in \phi_P(P)$, the subset sum dimension $d'$ of $A'-x$ is the same as the subset sum dimension of $A$, and furthermore $|P(A'-x)| \ge \gamma |P(A)|$. 
\end{definition}

In the above definition, note that $A'$ is a subset of $\phi_P(P)\subseteq \mathbb{Z}^{d(A)}$. Furthermore, by combining Lemma \ref{lem:shrink} and Lemma \ref{lem:decrease-dim}, we have that 
\[
    |\phi_P(P)|=|P| \ll_{\ell} |B|s(A)^{-\max(0,(d_1-\ell))} \ll |A|^{\beta},
\]
we have that $A'$ is a $(d(A), \beta + o_{|A|}(1))$-set.

\begin{lemma}\label{lem:reduce}
    Let $A$ be a non-averaging subset of a box $B\subseteq \mathbb{Z}^{\ell}$. 
    Assume that $|B| \le |A|^{\beta}$ for some $\beta >0$. Let $\epsilon, \delta, \gamma >0$ and assume that $\epsilon < 1/3$, $\gamma \le \delta^{K}$ for $K \ge C_{\beta,\epsilon}$ sufficiently large in $\beta,\epsilon$ and that $\gamma \ge |A|^{-1/3}$. We can find a non-averaging set $\tilde A$ with subset sum dimension $\tilde d$ and corresponding GAP $\tilde P = P(\tilde A)$ satisfying 
    \begin{align*}
        |\tilde P| &\ll_\beta |A|^{-(1-\epsilon)(\tilde d-\ell)} |B|, \qquad &&|\tilde A| > |A|^{1-\epsilon}, \qquad &&\text{if }~\tilde d > \ell,\\
        |\tilde P| &\ll_\beta \rho^{K} |B|, \qquad &&|\tilde A| = \rho |A| > |A|^{1-\epsilon}, \qquad &&\text{if }~\tilde d = \ell,\\
        |\tilde P| &\ll_\beta |B|, \qquad &&|\tilde A| > |A|^{1-\epsilon}, \qquad &&\text{if }~\tilde d < \ell,
    \end{align*}
    for some $\tilde d=O_\beta(1)$. 
    Furthermore, $\tilde A$ is $(\delta,\gamma)$-irreducible.
\end{lemma}

\begin{proof}
    We are going to construct the set $\tilde A$ via a series of moves which we call {\em down-move}, {\em up-move} and {\em shrink-move}. 
    Let $d = d(A)$ be the subset sum dimension of $A$ and $P = P(A)$ the $d$-bounding box of $\hat A \cup \{0\}$, where $\hat A \subset A$ is a set of size at least $|A| - c^{-1}s(A) \log |A|$ given by the definition of $d(A)$. 
    If $d > \ell$ then by Lemma \ref{lem:shrink} we have $|P| \ll s(A)^{-(d-\ell)} |B|$. If $d \le \ell$ then by Lemma \ref{lem:decrease-dim} we have $|P| \ll |B|$.
    
    Consider the set $\tilde{A} := \phi_P(\hat A)$ which is a subset of the box $\phi_P(P) \subset \mathbb{Z}^d$. Let $\tilde{d}$ be the subset sum dimension of $\tilde{A}$. We now consider three different ways in which the $(\delta,\gamma)$-irreducibility may be violated and in each of the cases replace the set $A$ with a certain shift of a subset of $\tilde A$. 
    
    Assume that $\tilde{A}$ has a subset $A_1$ with $|A_1| > \delta |A|$ and $A_1-x$ has subset sum dimension $d_1 < d$, for some $x \in \phi_P(P)$ (in particular, this happens when $\tilde{d} < d$). Then by Lemma \ref{lem:decrease-dim} we have $|P(A_1-x)| \ll |P|$. In this case, we say that $A_1-x$ is obtained from $A$ by a down-move.
    
    Assume that $\tilde{A}$ has a subset $A_1$ with $|A_1| > \delta |A|$ and $A_1-x$ has subset sum dimension $d_1 > d$ for some $x \in \phi_P(P)$ (in particular, this happens when $\tilde{d} > d$). By Lemma \ref{lem:shrink}, we have $|P(A_1-x))| \ll_d s(A_1)^{-(d_1-d)}|P|$. 
    In this case, we say that $A_1-x$ is obtained from $A$ by an up-move.

    Finally, consider the case when $\tilde{A}$ does not have any subset $A_1$ with $|A_1|>\delta |A|$ and $A_1-x$ has subset sum dimension $d_1\ne d$, but $\tilde{A}$ has a subset $A_1$ with $|A_1| > \delta |A|$ and $A_1-x$ has the same dimension $d_1 = d$ and a small bounding box $|P(A_1-x)| < \gamma |P|$. In this case, we say that $A_1-x$ is obtained from $A$ by a shrink-move.

    We will sequentially apply the down-move, up-move and shrink-move and replace $A$ by $A_1-x$ until we can no longer apply any move. In particular, we will stop the procedure when one of the following conditions holds: if the new set $A_1-x$ has size at most $|A|^{1-\epsilon/2}$, or if the resulting dimension $d_1$ exceeds an appropriate threshold $\overline{d}$ to be later chosen or, otherwise, if none of the 3 moves can be applied to $A$.
    
    Let $A_j$, $j=0, \ldots, i$, be the resulting sequence of sets starting with $A_0 = A$. Let $d_j$ be the subset sum dimension of $A_j$ and write $P_j = P(A_j)$. Note that all moves preserve the property of being non-averaging and hence $A_j$ is a non-averaging set for each $j\le i$. 
    
    Let $\mathcal{D}, \,\mathcal{U},\,\mathcal{S}$ be the set of down-moves, up-moves, shrink-moves respectively. We then have the following bound
    \begin{equation}\label{eq:upp-P}
        |P_i| < C^{|\mathcal{D}|+|\mathcal{U}|} \gamma^{|\mathcal{S}|} \left(\prod_{j\in \mathcal{U}} s(A_j)^{-(d_{j+1}-d_j)}\right) |P|.
    \end{equation}
    Here, $C$ is a constant that depends on $\max_j d_j$. 
    
    Furthermore, $d_i - d_0\le \sum_{j\in \mathcal{U}}(d_{j+1}-d_j) - |\mathcal{D}|$, and $|A_i| > \delta^i |A|$. Also note that $|\mathcal{S}|= i - |\mathcal{D}|-|\mathcal{U}|$, and $|P| < |A|^{O_{\beta}(1)}$. 
    
    We have $|A_i| > |A|^{1-\epsilon/2}$ and $d_i \le \overline{d}$ by design and so $s(A_j) \ge s(A_i) \ge |A|^{1-\epsilon}$ for sufficiently large $|A|$ and $C \le C_{\overline{d}}$, so (\ref{eq:upp-P}) implies that 
    \[
        |A|^{(1-\epsilon) \sum_{j\in \mathcal{U}} (d_{j+1}-d_j)} \le C^{|\mathcal D| + |\mathcal U|} |P| \le C_{\overline{d}}^{|\mathcal D| + |\mathcal U|} |P|.
    \]
    We have that $|\mathcal{D}| \le d_0-d_i + \sum_{j\in \mathcal{U}} (d_{j+1}-d_j) \le d_0 + \sum_{j\in \mathcal{U}} (d_{j+1}-d_j)$, and $|\mathcal{U}| \le \sum_{j\in \mathcal{U}} (d_{j+1}-d_j)$. Thus, 
    \[
        |A|^{(1-\epsilon) \sum_{j\in \mathcal{U}} (d_{j+1}-d_j)} \le C_{\overline{d}}^{d_0 + 2\sum_{j\in \mathcal{U}} (d_{j+1}-d_j)} |P|.
    \]
    Assuming $|A| > C_{\overline{d}}^{10d}$ and $\epsilon<1/2$, recalling that $|P| < |A|^{O_\beta(1)}$, it follows that 
    \begin{equation}\label{eq:Obeta}
    \max(d_i-d_0+|\mathcal D|, |\mathcal{U}|) \le \sum_{j\in \mathcal{U}} (d_{j+1}-d_j) = O_\beta(1).    
    \end{equation}
    We take $\overline{d}$ sufficiently large in $\beta$ so the threshold $d_j > \overline{d}$ is never reached by (\ref{eq:Obeta}). So (\ref{eq:upp-P}) now gives
    \[
    |P_i| \ll_\beta \gamma^{|\mathcal S|} |A|^{- (1-\epsilon) \sum_{j\in \mathcal U} (d_{j+1}-d_j)} |P|,
    \]
    i.e. 
    \begin{equation}\label{eq:upP2}
    |P_i| \ll_\beta \gamma^{i - |\mathcal U| -|\mathcal D|} |A|^{- (1-\epsilon) \sum_{j\in \mathcal U} (d_{j+1}-d_j)} |P|.
    \end{equation}
    Now suppose that the procedure stops because the next set $A_{i+1}$ has size less than $|A|^{1-\epsilon/2}$. It follows that 
    \[
    \delta^{i+1} |A| \le |A_{i+1}| \le |A|^{1-\epsilon/2}
    \]
    and so $\delta^{i+1} \le |A|^{-\epsilon/2}$. So if we take $\gamma \le \delta^{C_{\beta,\epsilon}}$ for large enough $C_{\beta,\epsilon}$ then we get $\gamma^{i-O_\beta(1)} \le |P|^{-10}$ contradicting (\ref{eq:upP2}). Thus, with the above choices of the parameters, the procedure cannot stop due to $|A_{i}|$ getting too small. Hence, it can only stop when none of the up-move, down-move or shrink-move can be applied to $A_i$. In particular, $A_i$ is $(\delta, \gamma)$-irreducible. Furthermore, we have $|A_i| \ge |A|^{1-\epsilon}$ and 
    \[
    |P_i| \ll_\beta |A|^{-(1-\epsilon)\sum_{j\in \mathcal U} (d_{j+1}-d_j)} |P|.
    \]
    
    Recall that we also have $|P| \ll |A|^{-(1-\epsilon) \max(0, d-\ell)} |B|$. If $d_i > \ell$ then this gives
    \[
        |P_i| \ll_\beta |A|^{-(1-\epsilon) (d_i - \ell)} |B|.
    \]
    If $d_i < \ell$ then this gives 
    \[
    |P_i| \ll_\beta |B|.
    \]
    If $d_i = \ell$ then using the assumptions $\gamma \ge |A|^{-1/3}$ and $\epsilon < 1/3$, (\ref{eq:upP2}) gives
    \[
    |P_i| \ll_\beta \gamma^{i} |B|
    \]
    which together with $|A_i| \ge \delta^i |A|$ and $\gamma \le \delta^K$ gives 
    \[
    |P_i| \ll_\beta \rho^K |B|, ~~\rho = |A_i|/|A|.
    \]
    Taking $\tilde A = A_i$ and $\tilde P = P_i$ concludes the proof.
\end{proof} 

\subsection{Intersecting subset sums of an irreducible set}

Recall that a set $A \subset \R^d$ is in {\em $\delta$-convex position} if for any $a \in A$ there is a half-space $H^+$ containing $a$ and such that $|A \cap H^+| \le \delta |A|$.

\begin{theorem}\label{thm:int-sums}
    Let $\ell, d\ge 1$ and $\beta >0$.
    Let $\gamma, \delta, \mu < 1$ be such that $\gamma \le \delta^{C}$ and $\delta \le \mu^C$ and assume that $\gamma \ge (\log |A|)^{-1/C'}$ for sufficiently large $C = C_{\beta, \ell}$ and $C' = C'_{\beta,\ell}$. Suppose that $A \subset \Z^\ell$ be a subset of a box $B$ with $|B| \le |A|^{\beta}$ with subset sum dimension $d$ and let $P = P(A)$ and $\hat A$ be the set given by Corollary \ref{cor:struc}.
    Suppose that $A$ is $(\delta,\gamma)$-irreducible and $\phi_P(\hat A)$ is not in $\mu$-convex position. Then for sufficiently large $|A|$ there is $a\in A$ and nonempty disjoint $\tilde{A}_1,\tilde{A}_2\subset A\setminus \{a\}$ such that 
    \begin{equation}\label{eq:thmintsums}
        \sum_{a_1 \in \tilde{A}_1} (a_1-a) = \sum_{a_2 \in \tilde{A}_2} (a-a_2). 
    \end{equation}
\end{theorem}

In the remainder of this subsection, we will always assume the setup of Theorem \ref{thm:int-sums}. Let $A_0 = \phi_P(\hat A)$ and $B = \phi_P(P)$.

For $a\in A_0$ such that there is no half-space $H^+$ containing $a$ satisfying $|A_0\cap H^+|\le \mu |A_0|$, we claim that there exists a convex combination $\sum_{x\in A_0} c_x(x-a) = 0$ for some choice of coefficients $0\le c_x \le \mu^{-1}/|A_0|, \sum_x c_x = 1$. Indeed, assume otherwise, noting that the set of points $\{\sum_{x\in A_0} c_x(x-a): 0\le c_x\le \mu^{-1}/|A_0|, \sum_x c_x=1\}$ is convex, we have a separating hyperplane defined by a non-zero vector $u \in \R^d$:
\begin{equation}\label{eq:duality}
\max_{c_x\in \left[0,\mu^{-1}/ |A_0|\right],~ \sum c_x=1} c_x (x-a)\cdot u < 0.    
\end{equation}
But if there are $m \ge \mu |A_0|$ many $x \in A_0$ satisfying $(x-a)\cdot u \ge 0$ then we can assign $c_x = \frac{1}{m}$ to those $x$ and 0 elsewhere, contradicting (\ref{eq:duality}).

Consider a convex combination $\sum_{x\in A_0}c_x(x-a)=0$ where $c_x\in [0,\mu^{-1}/|A_0|]$. Order the elements in $A_0 \setminus \{a\}$ as $a_1,\dots,a_m$ (where $m=|A_0|-1$) in increasing order of the coefficients $c_x$. We then let 
\begin{equation}\label{eq:defA12}
    A_1 = \{a_i: i=1\pmod 2\}, \qquad A_2 = \{a_i: i=0\pmod 2\}.
\end{equation}

\begin{lemma}\label{lem:gap-grid}
    Let $d = d(A) \ge 1$ and $A_1, A_2 \subseteq \mathbb{Z}^{d}$ be as above. The sets $A_1-a$ and $a-A_2$ have the subset sum dimension $d$. Let $P_1 = P(A_1-a)$ and $P_2 = P(a-A_2)$ for $i=1,2$. Then, for $i=1,2$,  $P_i$ is contained in a translate of $C B$, $P_i$ has size at least $|P_i|\ge \gamma |B|$, and, $\dim(P_i) = \dim(B) = d$. Here $C$ is a constant depending on $d$. 
\end{lemma}

\begin{proof}
    We have that $|A_1|, |A_2| \ge |A_0|/2 - 1 > |A|/4$. For $\delta < 1/4$, using that $A$ is $(\delta, \gamma)$-irreducible, we have that the subset sum dimension of $A_1-a$ is equal to $d$ so $\dim(P_1)=\dim(B)=d$, and the box $P_1 = P(A_1-a)$ has size at least $\gamma |B|$. Similarly for $A_2$ (noting that $A_2-a$ and $a-A_2$ have the same subset sum dimension).

    Next we show that $P_i$ is contained in a translate of $CB$ for a constant $C$ depending on $d$. Consider $i=1$ (the case $i=2$ is similar). The set $\{0\}\cup (A_1-a)$ is clearly contained in a translate $B-a$ of $B$. %
    Recall that there is a subset $A_1'$ of $A_1-a$ of size $s(A_1)$ with $\Sigma(A_1')$ containing a translate of $(cs(A_1)) P_1$. On the other hand, $\Sigma(A_1')$ is contained in a translate of $s(A_1) (P_1 \cap (B-a))$. Hence, $|s(A_1) (P_1 \cap (B-a))| \ge |(cs(A_1)) P_1|$, and thus $|P_1| \ll_d |P_1 \cap (B-a)|$. %
    The last inequality implies that there is a constant $C$ depending on $d$ such that $P_1$ is contained in a translate of $CB$, as claimed. 
    %Finally, since $P_i$ is a subset of $C B$ of density at least $\gamma$, by \cite[Lemma 2.16]{conlon2023homogeneous}, we have that $ctP_i$ contains a translate of $c'tB \cap \langle P_i\rangle$.
\end{proof}

We recall some useful results in \cite{conlon2023homogeneous}.

\begin{definition}
Given a finite subset $A$ of $\mathbb{Z}^{d}$, we define the zonotope ${\mathcal Z}_{A}$ to be the Minkowski sum of the segments $[0,1]\cdot a$ with $a\in A$. %
\end{definition}

The following lemma was proved in \cite{conlon2023homogeneous}. 
\begin{lemma} \label{lem:sampling}
Let $A$ be a subset of a box $Q$ in $\mathbb{Z}^{d}$ with widths $w_{1},\dots,w_{d}$ and $0\in A$ and let ${\mathcal Z}_{A}$ be the corresponding zonotope.
Then, for any $z\in{\mathcal Z}_{A}$, there exists a subset sum $s(z)$ of $A$ such that  $|z_{i}-s(z)_{i}|\le \sqrt{d|A|}w_{i}$ for all $i\le d$. 
\end{lemma}

We will apply Lemma \ref{lem:sampling} to the sets $A_1-a$ and $a-A_2$. %For $i=1,2$, let $z_i = \mu |A| \sum_{x\in A_i} c_x(x-a)$. 
Note that $\mu |A|c_x \in [0,1]$ for all $x$.  
\begin{lemma}\label{lem:zonotope}
    Let $c\in (0,1)$. Let $\xi = c'\mu \gamma$ for a suitable constant $c'$ depending on $c$ and $d$. Let $\overline{A}_1 \subset A_1$ be a set of size at least $c|A_1|$ and let $\overline{z}_1 = (\mu |A|/2) \sum_{x \in \overline{A}_1} c_x(x-a)$.
    Then we have $\overline{z}_1 + \xi |A|B \subseteq \mathcal{Z}_{A_1-a}$. Similarly, for $\overline{z}_2 = (\mu |A|/2) \sum_{x \in \overline{A}_2} c_x(a-x)$ we have $\overline{z}_2 + \xi |A|B \subseteq \mathcal{Z}_{a-A_2}$.
\end{lemma}
\begin{proof}
Let $\Delta_x = \min(\mu |A| c_x/2,1-\mu |A| c_x/2) = \mu |A| c_x/2$ (noting that $\mu |A|c_x \le 1$ for all $x \in A_0$), and define the centered zonotope $\overline{\mathcal{Z}}_{\overline{A}_1-a}$ as the Minkowski sum of the segments $[-\Delta_x,\Delta_x] (x-a)$ with $x\in \overline{A}_1$. We then have that 
\[
\overline{z}_1 + \overline{\mathcal{Z}}_{\overline{A}_1-a} \subseteq \mathcal{Z}_{\overline{A}_1-a} \subseteq \mathcal{Z}_{A_1-a}.
\] 

We show that $\overline{\mathcal{Z}}_{\overline{A}_1-a}$ contains $\xi|A|B$ for an appropriate $\xi>0$. Assume for the sake of contradiction that there is a point $x \in \xi|A|B \setminus \overline{\mathcal{Z}}_{\overline{A}_1-a}$. As $\overline{\mathcal{Z}}_{\overline{A}_1-a}$ is convex, there is $u\in \mathbb{R}^d$ and $z \in \R$ such that $u\cdot x > z$ whereas $u\cdot y < z$ for all $y\in \overline{\mathcal{Z}}_{\overline{A}_1-a}$. Since $0\in \overline{\mathcal{Z}}_{\overline{A}_1-a}$, we must have that $z\ge 0$. Furthermore, we obtain $\sum_{y\in \overline{A}_1-a} |u\cdot y| \Delta_y < z$. 

Pick $\delta' \le c_d\gamma$ for a suitable constant $c_d$ depending on $d$. Consider $H = \{y: |u\cdot y| < \delta' z/(\xi|A|)\}$. We claim that $|(\overline{A}_1-a)\cap H| \le 2\delta |\overline{A}_1|$. Indeed, assume that $|(\overline{A}_1-a)\cap H| \ge 2\delta |\overline{A}_1|$. Consider $H\cap \mathbb{Z}^d$. By Lemma \ref{lem:john}, there is a basis $\mathbf{v}$ and a tuple of integers $\mathbf{N}$ such that $P(\mathbf{v},\mathbf{N}) \supseteq H\cap B \supseteq P(\mathbf{v},c_d\mathbf{N})$. In particular, we have $|P(\mathbf{v},\mathbf{N})| \ll_d |H\cap B|$. Thus, we have a subset of $A_0$ of size at least $\delta |A|$ for which the $d$-bounding box has size at most $C_d |H\cap B|$. As $A$ is $(\delta,\gamma)$-irreducible, this implies that $|H\cap B| \ge C_{d}^{-1}\gamma |B|$.
However, we have $|H\cap B| \ll_d \delta' \cdot |B|$: to see this, note that the existence of the point $x\in \xi|A|B$ with $u\cdot x > z$ implies that $\max_i |u_i| w_i(B) > z/(\xi|A|d)$, so in each line in the $i$-th direction only $2d\delta'w_i(B)$ points may lie in $H$. As such, $|H\cap B| \le 2d\delta'|B|$. However, this is a contradiction for $\delta'$ chosen so that $2d\delta' < C_d^{-1}\gamma$. 

As such, the sum of the coefficients $c_x$ over $x\in H\cap (\overline{A}_1-a)$ is at most $|(\overline{A}_1-a)\cap H|\mu^{-1}/|A| \le 2\delta \mu^{-1} < 1/2$, for $\delta < \mu/4$. Thus, 
\[
    \sum_{x\in (\overline{A}_1-a)\setminus H} \Delta_x = (\mu |A|/2) \sum_{x\in (\overline{A}_1-a)\setminus H} c_x \ge \mu |A|/4.
\]
Hence, recalling that $|u\cdot y|\ge \delta' z/(\xi |A|)$ for $y\notin H$, 
\[
    \sum_{y\in \overline{A}_1-a} |u\cdot y| \Delta_y \ge \sum_{y\in (\overline{A}_1-a)\setminus H} \frac{\delta' z}{\xi |A|} \Delta_y\ge \frac{\delta' z}{\xi |A|} \cdot \frac{\mu |A|}{4} > z,
\] 
provided that we have chosen $\xi < \delta'\mu/4$, a contradiction. We conclude that $\xi |A| B \subset \overline{\mathcal Z}_{\overline{A}_1-a}$, which implies that 
\[
\overline{z}_1 + \xi|A|B \subset \overline{z}_1 + \overline{\mathcal Z}_{\overline{A}_1-a} \subset \mathcal Z_{A_1-a}.
\] 
An analogous argument shows that $\overline{z}_2 + \xi |A|B \subset \mathcal Z_{a-A_2}$.
\end{proof}

%\begin{lemma}\label{lem:intersect}
%    We have that $\Sigma(A_1-a) \cap \Sigma(a-A_2)\ne \emptyset$. 
%\end{lemma}
We now complete the proof of Theorem \ref{thm:int-sums}. 
\begin{proof}[Proof of Theorem \ref{thm:int-sums}]
    Recall the definition of $A_1, A_2 \subseteq A_0 = \phi_P(A)$ in (\ref{eq:defA12}), and $P_1 = P(A_1 - a)$, $P_2 = P(a-A_2)$. By definition, $A_1-a$ has a subset $\hat A_1$ of size at least $|A_1| - c^{-1}s(|A_1|)\log |A_1|$ contained in $P_1$ such that for $A_1'\subset \hat A_1$ of size at most $s_1 = s(A_1)$ we have that $\Sigma(A_1')$ contains $q_1 + cs_1P_1$, where $q_1 \in s_1P_1$. Here $c$ is a constant which can depend on $\ell$ and $\beta$. Let $\bar A_1 = \hat A_1\setminus A_1'$. We can similarly define sets $\bar A_2$, $A_2'$ and $s_2$.

    We will show that 
    \begin{equation}\label{eq:intersect-phi}
        (\Sigma(A_1-a) \cap \Sigma(a-A_2))\setminus \{0\}\ne \emptyset.
    \end{equation}
    
    By Lemma \ref{lem:sampling} for any $z\in \mathcal{Z}_{\bar A_1}$ there exists $\tilde z \in \Sigma(\bar A_1)$ and $r \in \sqrt{d |A_1|} P_1$ such that $z = \tilde z + r$. Note that $s_1 \gg \sqrt{|A_1|} $. Then it follows that for any $z\in \mathcal{Z}_{\bar A_1} \cap \langle P_1 \rangle$ and $t \in q_1 + (cs_1/2)P_1$ the sum $y = z+t \in  \langle P_1\rangle$ can be written as $y = z+t = \tilde z  + (r +t) \in  \Sigma(\bar A_1) + \Sigma(A'_1)$.
    Thus,
    \begin{equation}\label{eq:inclusion}
        \langle P_1 \rangle \cap (\mathcal{Z}_{\bar A_1} +q_1 + (cs_1/2) P_1) \subset \Sigma(\bar A_1) + \Sigma(A_1') \subset \Sigma(A_1-a),
    \end{equation}
    and similarly, we have 
    \[
    \langle P_2 \rangle \cap (\mathcal{Z}_{\bar A_2} + q_2 + (cs_2/2) P_2) \subset \Sigma(a- A_2).
    \]

    Now denote $z_1 = \mu |A| \sum_{x \in A_1} c_x (x-a)$ and $z_2 = \mu |A| \sum_{x \in A_2} c_x (a-x)$. Note that, in fact, we have $z_1 = z_2$ since by the choice of coefficients $c_x$ we have $\sum_{x\in A} c_x (x-a) = 0$. We can write $z_1 = \bar z_1 + z_1'$ and $z_2 = \bar z_2 + z_2'$, where $\bar z_1 = \mu |A| \sum_{x \in \bar A_1} c_x (x-a)$ and $\bar z_2 = \mu |A| \sum_{x \in \bar A_2} c_x (a-x)$. 
    By Lemma \ref{lem:zonotope}, we have
    \[
    \conv (\bar z_1 + \xi |A| B) \subset \mathcal Z_{A_1-a} ~\text{ and }~ \conv (\bar z_2 + \xi |A| B) \subset \mathcal Z_{a-A_2}.
    \]
    (Recall that $B$ is an intersection of a convex set with the integer grid.) 
    We thus have that $\Sigma(A_1-a)$ contains 
    \[
        \langle P_1\rangle \cap (q_1 + \mathrm{conv}(\bar z_1 + \xi |A|B) + (cs_1/2)P_1).
    \]
    Similarly, $\Sigma(a-A_2)$ contains 
    \[
        \langle P_2\rangle \cap (q_2 + \mathrm{conv}(\bar z_2 + \xi |A|B) + (cs_2/2)P_2).
    \]
    Let $R_1 = A_1'\cup ((A_1-a)\setminus \hat A_1)$ and $R_2 = A_2'\cup ((a-A_2)\setminus \hat A_2)$ and note that $|R_i| \le O(s(A_i)\log |A_i|)$ since $|\hat{A}_i| \ge |A_i| - c^{-1}s(A_i)\log |A_i|$ and $|A_i'| \le s(A_i)$. Thus we have that $\bar z_1 - z_1 = \sum_{x\in A_1 \setminus a+\overline{A}_1} \mu |A|c_x (x-a)\in \mathrm{conv}((Cs(A)\log |A|)B)$ and similarly $\bar z_2-z_2 \in \mathrm{conv}((Cs(A)\log |A|)B)$. Also recall that $z_1=z_2$.
    Furthermore note that $q_1\in s_1P_1$ and hence $\langle P_1\rangle \cap (q_1 + \mathrm{conv}(\bar z_1 + \xi |A|B) + (cs_1/2)P_1)$ contains $\langle P_1\rangle \cap (z_1 + \xi' |A| B)$ for an appropriate $\xi' = \xi/2$. Indeed, $z_1 - \overline{z}_1 \in \mathrm{conv}((Cs(A)\log |A|)B)$ and $-q_1 \in s_1 P_1$ so \[z_1 \xi'|A|B \subseteq q_1 + \overline{z}_1 + \mathrm{conv}((\xi'|A|+Cs(A)\log |A|)B) + s_1P_1 \subseteq q_1+\overline{z}_1+\mathrm{conv}(\xi|A|B).\] Here by the definition of $s(A)$ and $\xi = c'\mu \gamma$, under the assumptions on $\mu$ and $\gamma$, we have that $s(A)\log |A| = |A|/\log |A| < \xi |A|/2C$. Hence, it suffices to show that 
    \[
        (\langle P_1\rangle \cap (z_1 + \xi' |A| B)) \cap (\langle P_2\rangle \cap (z_2 + \xi' |A| B)) \ne \{z_1\}.
    \]
    This must hold since $\langle P_1\rangle \cap \langle P_2\rangle$ is a lattice of bounded covolume in $\langle B\rangle$ by Lemma \ref{lem:gap-grid}. 

    This establishes (\ref{eq:intersect-phi}). In particular, there exists nonempty $\tilde{A}_i' \subseteq A_i \subseteq \mathbb{Z}^d$ such that 
    \[
        \sum_{x\in \tilde{A}_1'} (x-a) = \sum_{x\in \tilde{A}_2'} (a-x).
    \]
    Under the map $\phi_P^{-1}$, we then have, for $\tilde{A}_1 = \phi_P^{-1}(\tilde{A}_1')$ and $\tilde{A}_2 = \phi_P^{-1}(\tilde{A}_2')$ and $\tilde{a} = \phi_P^{-1}(a)\in A$, that
    \[
        \sum_{x\in \tilde{A}_1'} (x-\tilde{a}) = \sum_{x\in \tilde{A}_2'} (\tilde{a}-x),
    \]
    yielding the desired conclusion (\ref{eq:thmintsums}). 
\end{proof}
%Lemma \ref{lem:intersect} readily implies Theorem \ref{thm:int-sums}. 

\section{Non-averaging sets}\label{section:non-averaging}
Recall that that $A\subset \mathbb{Z}^d$ is \emph{non-averaging} if there is no element $a$ in $A$ which can be written as an average of a subset $A'$ not containing $a$. In this section, we will give the proof of Theorem \ref{thm:main}, and the special case Theorem \ref{thm:main-0}. For convenience of the reader, we recall the statement of Theorem \ref{thm:main}. 

\begin{theorem*}%
    Let $A \subseteq \mathbb{Z}^d$ be a non-averaging subset of a box $B$. Then $|A| \le |B|^{\alpha_d+o(1)}$ where $\alpha_d = (d-1)/(d+1)$ for $d\ge 2$ and $\alpha_1 = 1/4$. 
\end{theorem*}

We will need the following numerical fact:

\begin{obs}\label{obs:numerics}
    For any $\zeta \in (0, 1)$, there exists $c(\zeta) >0$ depending continuously on $\zeta$ such that: 
    \[
    (\alpha_{\hat d}+\zeta) \left(\frac{1}{\alpha_d+\zeta} - (1-\epsilon)(\hat d-d)\right) \le 1-c(\zeta),
    \]
    for any $\hat d > d \ge 1$ and $\epsilon \le 0.01 \zeta$.
\end{obs}

\begin{proof} Indeed, for $\hat d > d > 1$, 
\begin{align*}
    (\alpha_{\hat d}+\zeta)\left(\frac{1}{\alpha_d+\zeta} - (1-\epsilon)(\hat d -d)\right) &= \left(\frac{\hat d -1}{\hat d + 1}+\zeta\right)\left(\frac{1}{\frac{d-1}{d+1}+\zeta} - (1-\epsilon)(\hat d -d)\right)\\
    &< 1-c.
\end{align*}
For $\epsilon \le \zeta < 0.01$ and for $c\le \zeta$, it is a simple numerical check to show that the above inequality holds for all $\hat d > d > 1$. Indeed, for the left hand side to be non-negative, as $d\ge 2$, we need $1\le \hat{d}-d \le 3$. For $d\ge 7$, we have 
\begin{align*}
    \left(\frac{\hat d -1}{\hat d + 1}+\zeta\right)\left(\frac{1}{\frac{d-1}{d+1}+\zeta} - (1-\epsilon)(\hat d -d)\right)& \le \left(\frac{\hat d -1}{\hat d + 1}+\zeta\right) \left(\frac{4}{3} - (1-\epsilon)\right) \\
    &\le 2 \left(\frac{1}{3}+\epsilon\right) < 1-c.
\end{align*}
A finite check over $d<7$ and $d<\hat{d}\le d+3$ establishes the desired claim.

If $d=1$, then similarly, for $\hat d > d = 1$, and $\epsilon$ sufficiently small in $\zeta$, 
\begin{align*}
    (\alpha_{\hat d}+\zeta)\left(\frac{1}{\alpha_d+\zeta} - (1-\epsilon)(\hat d -d)\right) &= \left(\frac{\hat d -1}{\hat d + 1}+\zeta\right)\left(\frac{1}{1/4+\zeta} - (1-\epsilon)(\hat d - 1)\right) \\
    &< 1-c,
\end{align*}
for $\epsilon \le c_1 c$ and $c\le c_2\zeta$ for appropriate absolute constants $c_1, c_2>0$. Here, we again note that the left hand side is only non-negative when $\hat{d} \le 5$ and a finite check establishes the claim. The critical case leading to the dependence between $\epsilon,\,c$ and $\zeta$ occurs when $\hat{d}=2$ and $\hat{d}=3$.   
\end{proof}

We remark that the critical case $d=1, \hat{d}=2$ also reflects the tightness of the non-averaging set problem witnessed by the construction in dimension $2$. Also, note that $\alpha_d\ge 1/4$ for all $d\ge 1$ and $\alpha_d$ is increasing in $d$. 

We will next complete the proof of Theorem \ref{thm:main}. In particular, we want to show that, for any $\zeta>0$, no non-averaging set $A\subseteq B\subseteq \mathbb{Z}^d$ of size $|A| > |B|^{\alpha_d+\zeta}$ exists for sufficiently large $|A|$. By applying Lemma \ref{lem:reduce}, we show that we replace $A$ by a non-averaging irreducible set $\tilde{A}$ which is a subset of a box $\tilde{P}$ where $|\tilde{A}| \ge |\tilde{P}|^{\alpha_{\tilde{d}} + \tilde{\zeta}}$ where $\tilde{\zeta}-\zeta \ge \theta(\zeta,d) > 0$ whenever $\dim(\tilde{A})\ne \dim(A)$ and otherwise $\tilde{\zeta}-\zeta \ge \theta'(\zeta,d,|A|)>0$. Such increment (in the $\tilde{\zeta}$ parameter) follows quickly when $\tilde{d} \ne d$. When $\tilde{d}$, the increment follows from the fact that non-averaging $\tilde{A}$ must be in $\mu$-convex position, together with the density increment Lemma \ref{lem:convex1}. As the increment depends on the current dimension parameter $d$, we need to ensure that along this iteration that the maximum dimension of the sets $\tilde{A}$ stay bounded by a suitably chosen parameter $D$, and the sets $\tilde{A}$ remain large (relative to the original set $A$). Given these conditions, we obtain a contradiction, as the parameter $\tilde{\zeta}$ is bounded above by $1$. 

\begin{proof}[Proof of Theorem \ref{thm:main}]
Consider $A\subseteq B \subseteq \mathbb{Z}^d$ with $|A| > |B|^{\alpha_d + \zeta}$ for some $\zeta >0$. Let $D=D(d, \zeta)$ be sufficiently large in $d$ and $\zeta$.   
Let $\epsilon \le 0.01 \zeta$ be small enough depending on $\zeta$, $D$ and $d$. %
Let $K=K(\epsilon)$ be large enough depending on $\epsilon$. We then take $\kappa$ to be sufficiently small depending on $\epsilon$ and $K^{-1}$, $\delta = (\log \log |A|)^{-\kappa}$, $\gamma = \delta^K$ and $\mu = \delta^{\kappa}$. We assume that $|A|$ is sufficiently large in the parameters $\epsilon$ and $d$. For simplicity of notation, throughout the proof we write $\ll$ for $\ll_d$. %

We now apply Lemma \ref{lem:reduce} to replace $A$ by a non-averaging set $\tilde{A}$ which is $(\delta,\gamma)$-irreducible. We consider three cases depending on the dimension $\tilde{d}$ of $\tilde{A}$ given by Lemma \ref{lem:reduce}. In the case $\tilde{d}\ne d$, based on $\tilde{A}$, we show that we can deduce a non-averaging subset $A'$ of $\mathbb{Z}^{\tilde{d}}$ which is a subset of a box $B'$ with $|A'|\ge |B'|^{\alpha_{\tilde{d}} + \zeta + \eta(\zeta,d)}$ where $\eta(\zeta,d)>0$. In the case $\tilde{d}=d$, we will obtain a non-averaging subset $A'$ of a box $B'$ which satisfies $|A'| \ge |B'|^{\alpha_{\tilde{d}}+\zeta+\theta(\zeta,d,|A|)}$ where $\theta(\zeta,d,|A|)>0$. %

{\noindent \bf Case 1:} $\tilde{A}$ has dimension $\tilde d > d$. 

%First consider the case that we can replace $A$ by $\tilde A$ with dimension $\tilde d > d$ with
Let the bounding box $\tilde P = P(\tilde A)$, then 
\[
    |\tilde P| \ll |A|^{-(1-\epsilon)(\tilde d -d)} |B| ~\text{ and }~ |\tilde A| \gg |A|^{1-\epsilon},
\]
for some $\tilde d = O_d(1)$. 
Note that 
\begin{align*}
    |\tilde P|^{\alpha_{\tilde d} + \zeta} & \ll (|A|^{-(1-\epsilon)(\tilde d -d)}|B|)^{\alpha_{\tilde d} + \zeta} \\
    &\ll ( |A|^{-(1-\epsilon)(\tilde d -d)} |A|^{\frac{1}{\alpha_d+\zeta}})^{\alpha_{\tilde d} + \zeta},
\end{align*}
Hence, by Observation \ref{obs:numerics},
\begin{equation}\label{eq:incr-1}
    |\tilde A| \gg |A|^{1-\epsilon} \gg |\tilde P|^{(\alpha_{\tilde{d}}+\zeta)(1-\epsilon)/(1-c(\zeta))} \gg |\tilde{P}|^{\alpha_{\tilde{d}} + \zeta + c(\zeta/2)}.
\end{equation}
Here, we assume that $\epsilon$ is sufficiently small in $c(\zeta)$. Thus, in this case, we can replace $A \subset B \subset \Z^d$ and $\zeta>0$ with $\phi_{\tilde P}(\hat A) \subset \phi_{\tilde P}(\tilde P) \subset \Z^{\tilde d}$ and $\tilde \zeta = \zeta+c(\zeta)$.% and iterate. 

{\noindent \bf Case 2:} $\tilde{A}$ has dimension $\tilde d < d$. 

Then for $\tilde A$ and $\tilde P = P(\tilde{A})$ we have
\[
    |\tilde P| \ll |B| ~\text{ and }~ |\tilde A| \gg |A|^{1-\epsilon},
\]
and so 
\begin{align}\label{eq:incr-2}
|\tilde A| \gg |A|^{1-\epsilon} \gg |B|^{(1-\epsilon) (\alpha_d +\zeta)} \gg |\tilde P|^{\alpha_{\tilde d} +\zeta + (\alpha_d-\alpha_{\tilde d}-\epsilon)}
\end{align}
so we can replace $A \subset B \subset \Z^d$ and $\zeta >0$ with $\phi_{\tilde P}(\hat A) \subset \phi_{\tilde P}(\tilde P) \subset \Z^{\tilde d}$ and $\tilde \zeta = \zeta + (\alpha_d-\alpha_{\tilde d}-\epsilon)$. Here, we assume that $\epsilon$ is chosen so that $\epsilon < \alpha_d - \alpha_{\tilde{d}}$ and hence $\tilde{\zeta} \ge \zeta + \theta(d)$ for some $\theta(d)>0$ (here we note that $\tilde{d}$ is bounded above as a function of $d$).  

{\noindent \bf Case 3:} $\tilde{A}$ has dimension $\tilde d = d$.

In this case, for $\tilde A$ and $\tilde P = P(\tilde{A})$ we have 
\[
    |\tilde P| \ll \rho^K |B| ~\text{ and }~ |\tilde A| = \rho |A| \gg |A|^{1-\epsilon},
\]
We also have that $\tilde A$ is $(\delta,\gamma)$-irreducible. Let $\hat A \subset \tilde A$ be the subset given by Theorem \ref{thm:struc}. Let $\overline{B} = \phi_{\tilde P}(\tilde P)$ and $\overline{A} = \phi_{\tilde P}(\hat A)$. We have $|\overline{B}| \ll \rho^K |B|$ and $|\overline{A}| \gg \rho |A|$.
If $\overline{A} \subset \overline{B}$ is not in $\mu$-convex position, we obtain a contradiction with Theorem \ref{thm:int-sums} as $\tilde A$ is non-averaging. Thus, $\overline{A}$ is in $\mu$-convex position. By Lemma \ref{lem:convex1}, there is some $\eta \in [\mu, \mu^\tau]$ for $\tau = \tau(d, \epsilon)>0$ and a convex subset $\Omega' \subset \operatorname{conv}(\overline{B})$ with $\mathrm{Vol}(\Omega') \le \eta |\overline{B}|$ and $|\Omega'\cap \overline{A}| \ge \eta^{\frac{d-1}{d+1}+\epsilon}|\overline{A}|$. By Lemma \ref{lem:john}, there is a box $\tilde{B}$ in $\mathbb{Z}^{d}$ such that $|\tilde{B} \cap \overline{A}| \ge \eta^{\frac{d-1}{d+1}+\epsilon}|\overline{A}|$ and $|\tilde{B}| \ll \eta |\overline{B}|$. Note that 
\begin{align}\label{eq:up-tildeB}
    |\tilde{B}|^{\alpha_{d} + \zeta} &\le (\eta |\overline{ B}|)^{\alpha_{d} + \zeta} \ll \eta^{\alpha_d+\zeta} \rho^{K/4}|B|^{\alpha_{d} + \zeta} \ll \eta^{\alpha_d+\zeta} \rho^{K/4} |A|,
\end{align}
while 
\[
    |\tilde{B} \cap \overline{ A}| \ge \eta^{\frac{d-1}{d+1}+\epsilon} \rho |A|.
\]
We have $\alpha_d + \zeta - \frac{d-1}{d+1}-\epsilon \ge \zeta/2 > 0$, and we can take $K \ge 100$, hence
\begin{equation}\label{eq:tildeB-barA}
    |\tilde{B} \cap \overline{A}| > \eta^{-\zeta/2} \rho^{-K/10} |\tilde{B}|^{\alpha_{\tilde d}+\zeta}.
\end{equation}
%We have
%\begin{equation}\label{eq:est}
%\eta^{-\zeta/2} \ge \delta^{-\tau(d, \epsilon) \zeta/2} \ge |A|^{\kappa \tau(d, \epsilon) \zeta/2}.%
%\end{equation}
Thus we can replace $A\subset B\subset \Z^d$ and $\zeta >0$ with $\tilde{B}\cap \overline{A} \subset \tilde{B}$ where $|\tilde{B}\cap \overline{A}| \ge |\tilde{B}|^{\tilde{\zeta}}$ and $\tilde{\zeta} \ge \zeta + \theta'(\zeta,d,|A|)$ for $\theta'(\zeta,d,|A|)>0$. % (combining (\ref{eq:up-tildeB}), (\ref{eq:est}), (\ref{eq:tildeB-barA})  and the assumption that $|A|$ is sufficiently large) and iterate the argument.

$\qquad$

In each of the three cases, we replace $A$ by a set $A'$ inside a box $B' \subseteq \mathbb{Z}^{\tilde{d}}$ such that $|A'| \ge |B'|^{\alpha_{\tilde{d}} + \zeta + \theta(\zeta,d)}$ when $\tilde{d}\ne d$ and $|A'| \ge |B'|^{\alpha_{\tilde{d}} + \zeta + \theta'(\zeta,d,|A|)}$ when $\tilde{d}=d$ and iterate. Let $A_j\subset B_j \subset \Z^{d_j}$ and $\zeta_j$ be the sets and parameter after the $j$-th iteration. We stop the iteration as soon as the dimension $\tilde{d}$ exceeds $D$, or $|A_j| < |A|^{1-c}$ for an absolute constant $c>0$. We next show that for a suitable choice of $D$, we can guarantee that the iteration never stops due to $\tilde{d}\ge D$ or $|A_j| < |A|^{1-c}$. %, and the number of iterations is bounded by a constant depending on $\zeta, d, \epsilon, K, \kappa$. 
For $\epsilon$ sufficiently small in $\zeta$, $D$ and $d$, we can guarantee that in all cases above $\zeta_j$ can only increase along the iteration. Furthermore, note that $\inf_{\zeta' \in [\zeta,0.9]} c(\zeta) = c'(\zeta) > 0$, and once the value of $\zeta_j$ exceeds $0.9$ we have $|A_j| \ge |B_j|^{1.1} > |B|$ which is a contradiction. This implies that the number of steps $j$ where $d_{j} > d_{j-1}$ is at most $c'(\zeta)^{-1} = O_\zeta(1)$. Since in each step where $d_{j}>d_{j-1}$ we have $d_j = O_{d_{j-1}}(1)$, we also obtain that $\max_j d_j = O_{d,\zeta}(1)$ along the iteration. Thus, by choosing $D$ sufficiently large depending only on $d$ and $\zeta$, we can guarantee that no $j$ along the iteration allows $\tilde{d}$ to exceed $D$. We also obtain that the number of iterations $j$ where $d_{j}<d_{j-1}$ is bounded by $O_{d,\zeta}(1)$. 

%Given the fixed choice of $D$, by choosing $\epsilon$ sufficiently small in $\zeta$, $D$ and $d$ and then choosing $K$ and $\kappa$ appropriately, we can guarantee that in each iteration the increase $\zeta_{j}-\zeta_{j-1}$ is at least a constant $\iota>0$ depending only on $\zeta, d, \epsilon$ and $\kappa$. As such, the number of iterations is at most $\iota^{-1}$ which is a constant depending on $\zeta, d, \epsilon, K, \kappa$. 

Finally, we verify that $|A_j|\ge |A|^{1-c}$ for all $j$, as such our choice of $\kappa$ can guarantee $\gamma \ge |A_j|^{-c}$ for an absolute constant $c>0$. Let $\rho_j = |A_{j+1}|/|A_j|$. We have $\rho_j \ge |A|^{-\epsilon}$ for all $j$. 
Since the number of steps $j$ where $d_{j+1}\ne d_{j}$ is $O_{\zeta, d}(1)$ as shown above, letting $\mathcal{J}$ denote the collection of such steps, we have 
\[
\prod_{j\in \mathcal{J}} \frac{|A_{j+1}|}{|A_j|} \ge |A|^{- O_{\zeta,d}(\epsilon)}
\]
which can be made larger than $|A|^{-c/2}$ if we take $\epsilon = \epsilon(\zeta, d)$ sufficiently small. Over steps $j$ where $d_{j+1}=d_j$, we have 
\[
|B_{j+1}| \ll \rho_j^{K} |B_j|.
\]
Thus, we have 
\[
    1 \le |B| \prod_{j \notin \mathcal{J}} \frac{|B_{j+1}|}{|B_j|} \le |A|^{C}  \prod_{j \notin \mathcal{J}} \max(1,C\rho_j^{K}),
\]
for a constant $C=C(D)$. Finally, we note that $\rho_j \le \mu^{c(d)}$ and hence for sufficiently large $|A|$, $C\rho_j^K \le \rho_j^{K/2}$. Hence,
\[
    \prod_{j\notin \mathcal{J}} \rho_j \ge |A|^{-2C/K}. 
\]
We then obtain that 
\[
    \prod_{j\notin \mathcal{J}} \rho_j \ge |A|^{-c/2}.
\]
As such, we obtain that $|A_j| \ge |A|^{1-c}$ for all $j$, proving the desired claim. This gives the desired contradiction. \qedhere

\end{proof}

\appendix

\section{Proof of Theorem \ref{thm:struc}} \label{section:appendix}

The main structural theorem of Conlon, Fox and the first author \cite{conlon2023homogeneous} is as follows. 
\begin{theorem}[\cite{conlon2023homogeneous}]\label{thm:sumset}
For any $\beta > 1$, and $0 < \eta < 1$, there are positive constants $c$ and $d$ such that the
following holds. Let $A$ be a subset of $[n]$ of size $m$ with $n \le m^\beta$ and let $s \in [m^\eta, cm/ \log m]$. Then there exists a subset $\hat A$ of $A$ of size at least $m-c^{-1} s \log m$, a proper GAP $P$ of dimension at most $d$ such that $\hat A \cup \{0\}$ is contained in $P$ and a subset $A'$ of $\hat A$ of size at most $s$ such that $\Sigma(A')$ contains a homogeneous translate of $csP$, where $csP$ is proper.
\end{theorem}
Theorem \ref{thm:struc} follows directly in a verbatim manner from their proof. For convenience of the reader, we record here a different self-contained derivation of Theorem \ref{thm:struc}, which applies to subsets of $\mathbb{Z}^{\ell}$, assuming Theorem \ref{thm:sumset}.

Fix parameters $\ell$, $\beta > 1$, and $0 < \eta < 1$ and let $A \subset B \subseteq \mathbb{Z}^\ell$ be a set of size $m$ with $|B| \le m^\beta$ and let $s \in [m^\eta, cm/ \log m]$. Without loss of generality, we can assume that $B = [n]^\ell$ upon changing $\beta$ to $\beta \ell$ (by taking $n$ to be the maximum sidelength of $B$, noting that $n^\ell \le |B|^\ell \le m^{\beta \ell}$). 

Let $\kappa >1$ and denote $H = n^{\kappa}$, $n_0 = H^\ell$. Consider the injective map $\phi:(-H/2,H/2]^\ell \rightarrow \mathbb{Z}$ given by
\[
\phi(a_1, \ldots, a_\ell) = \sum_{j=1}^\ell H^{j-1} a_j.
\]
Note that $\phi([n]^\ell) \subseteq [n_0]$. 
Let $A_0 = \phi(A)$, $\beta_0 = \beta \kappa \ell$ and note $m^{\beta_0} \ge n_0$. By Theorem \ref{thm:sumset} we get a GAP $P_0 \subset \Z$ of dimension $d\le d_0 = d(\beta_0, \eta)$ and a subset $\hat A_0 \subset A_0$ of size at least $m - c^{-1} s \log m$ and a subset $A'_0 \subset \hat A_0$ of size at most $s$. We have $\hat A_0 \cup\{0\} \subset P_0$ and $x_0+csP_0 \subset \Sigma(A'_0)$ for a homogeneous translate $x_0+csP_0$ and $csP_0$ is proper.

Let $\hat A$ and $A'$ be the preimages of $\hat A_0$ and $A'_0$ under $\phi$, respectively. Let $Q = \phi([-n, n]^\ell)$, this is an $\ell$-dimensional homogeneous GAP and $(\frac{cH}{n}) Q$ is proper. Note
\[
x_0+csP_0 \subset \Sigma(A_0') \subset sQ \cap sP_0.
\]
So $x_0 \in sQ$ and $csP_0 \subset 2s Q$. Write $P_0 = \{\sum_{i=1}^d n_i q_{0i}:~ a_i \le n_i \le b_i\}$ where $a_i \le 0 \le b_i$ (see \cite[Claim 2.6]{conlon2023homogeneous}). It follows that $q_{0i} \in 2s Q$, $i=1, \ldots, d$. Let $q_i = \phi^{-1}(q_{0i}) \in [-2sn, 2sn]^\ell$. We have
\[
\operatorname{Vol}(P_0) \le |\Sigma(A'_0)| \le (sn)^\ell,
\]
so $b_i-a_i \le (sn)^\ell$ for all $i=1, \ldots, d$.
If $cH > ns^2 (sn)^\ell$, %
we get that 
\[
P = \left\{ \sum_{i=1}^d n_i q_i:~ a_i \le n_i \le b_i\right\}
\]
is a homogeneous $d$-dimensional GAP such that $\phi(P) = P_0$ and $csP$ is proper and contained in $(\frac{cH}{n}) Q$. It follows that $\hat A \cup \{0\} \subset P$ and $x+csP \subset \Sigma(A')$ where $x = \phi^{-1}(x_0) \in [-s,s]^\ell$. Taking $\kappa = 10\ell^3$ ensures the desired lower bound on $H$ and completes the proof of Theorem \ref{thm:struc}.

\end{document}